\documentclass{amsart}

\usepackage{Preset}
\usepackage[makeroom]{cancel}
\usepackage{MnSymbol}
\usepackage{nccmath}

\usepackage{mathrsfs}
\usepackage[latin1]{inputenc}
\usepackage[T1]{fontenc}

\usepackage{tikz}
\usepackage[colorlinks=true,linkcolor=blue!50!black,anchorcolor=red,citecolor=blue,filecolor=black,menucolor=black,runcolor=black,urlcolor=black]{hyperref}

\usetikzlibrary{decorations.pathmorphing}
\usetikzlibrary{shapes,arrows,graphs}

\title[Uniform bounds for bubbles]{Uniform bounds for bubbles of neutral quadratic polynomials}
\author{Dzmitry Dudko}  
\email{dzmitry.dudko@stonybrook.edu}  
\author{Willie Rush Lim}
\email{willie\_rush\_lim@brown.edu}
\date{}

\begin{document}

\begin{abstract}
    Given a quadratic polynomial with an irrationally indifferent fixed point of bounded type, a bubble is an iterated preimage of its Siegel disk. 
    Similar to pseudo-Siegel disks, one constructs pseudo-bubbles from bubbles by filling in parabolic fjords. We introduce the near-degenerate regime for the $\alpha'$-Dynamics controlling the geometry of these pseudo-bubbles in deep scale. Consequently, we obtain uniform bounds on the size and regularity of pseudo-bubbles independent of the rotation number. In~\cite{DLL}, jointly with Lyubich, this result serves as a necessary ingredient for establishing uniform butterfly bounds and ultimately the combinatorial rigidity of the full attractor of neutral renormalization. 
   
\end{abstract}

\maketitle

\setcounter{tocdepth}{1}
\tableofcontents

\section{Introduction}
\label{sec:intro}

Consider the Neutral Family
\begin{equation}
    \label{eq:NeutrFamily}
    f_\theta(z) := e^{2\pi i \theta} z + z^2, \qquad \theta \in \R/\Z.
\end{equation}
This is the simplest family of global non-linear holomorphic dynamical systems that admit a neutral fixed point at $0$.
It exhibits some of the most delicate behavior in one-dimensional complex dynamics depending on the arithmetic properties of the rotation number $\theta$. The active study of the Neutral Family emerged in the 1980s; see, for instance~\cite{D87, McM98, Shi98, PZ04, IS, DLS, SY24, Che25} and references therein. The current paper further develops the unified renormalization theory of~\eqref{eq:NeutrFamily} originated in~\cite{DL22}.    

As Figure~\ref{fig:bubble-example} illustrates, for bounded-type $\theta$, the dynamical plane of $f_\theta$ contains the rotating Siegel disk $Z_\theta$ and infinitely many \emph{bubbles}, which are iterated preimages of $\overline Z_\theta$. The renormalization theory of $f_\theta\colon \overline Z_\theta\selfmap$ is developed through the first return maps $f^{\qq_n}_\theta$, where $\pp_n/\qq_n\approx \theta$ are the best rational approximations. In~\cite{DL22}, appropriate uniform precompactness of the family $\{f_\theta^{\qq_n(\theta)}\colon \overline Z_\theta\selfmap\}_{\theta,n }$ was established. In this paper, we extend this uniform geometric control to the bubbles of $f_\theta$. The control of bubbles is a necessary ingredient for comprehensive development of the neutral renormalization theory, including rigidity and hyperbolicity; see~\S\ref{ss:applications}. In particular, as Figure~\ref{fig:butterfly} illustrates, geometric control of the bubbles is required for the butterfly structure. 

The main results (Theorem~\ref{main-thm-01}) are obtained by introducing the \emph{near-degenerate regime} for what we refer to as the \emph{$\alpha'$-Dynamics}. This dynamics is given by the iterates $f_\theta^{\qq_n-\qq_{n-1}}$, which provide a telescope-type mechanism to elevate bubbles from the deeper to shallower geometric levels; see Figure~\ref{fig:bubble-zoom}. The $\alpha'$ points are the centers of bubbles and form the hyperbolic repelling set under the $\alpha'$-Dynamics; see~\S\ref{sss:alpha'.Dynam} for details.

To the best of our understanding, the hyperbolic contraction of $f_\theta^{-1}$ in $\C\setminus \overline Z_\theta$ alone does not suffice to extend the geometric bounds established in~\cite{DL22} from $f_\theta\colon\overline Z_\theta\selfmap$ to the bubbles, see~\S\ref{sss:butterfly-rigidity} and also~\S\ref{sss:hoglet.singul} for details. This limitation motivates our use of the near-degenerate regime. Since its introduction by Kahn~\cite{K06} in the mid-2000s, the near-degenerate regime has been employed in a variety of settings, leading in each case to substantial progress on the corresponding problems; see~ \cite{KL09,KL09a, KL08,KL09b,DL22,DL26a,Lim23,DLuoLyu25,DLuo25,KLL26}.

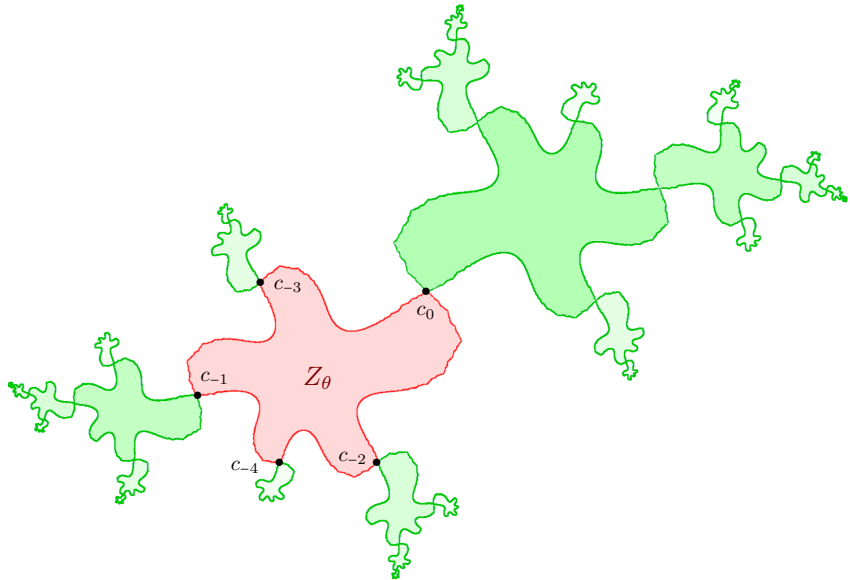
\begin{figure}
    \centering
\begin{tikzpicture}
    \node[anchor=south west,inner sep=0] (image) at (0,0) {\includegraphics[width=0.9\linewidth]{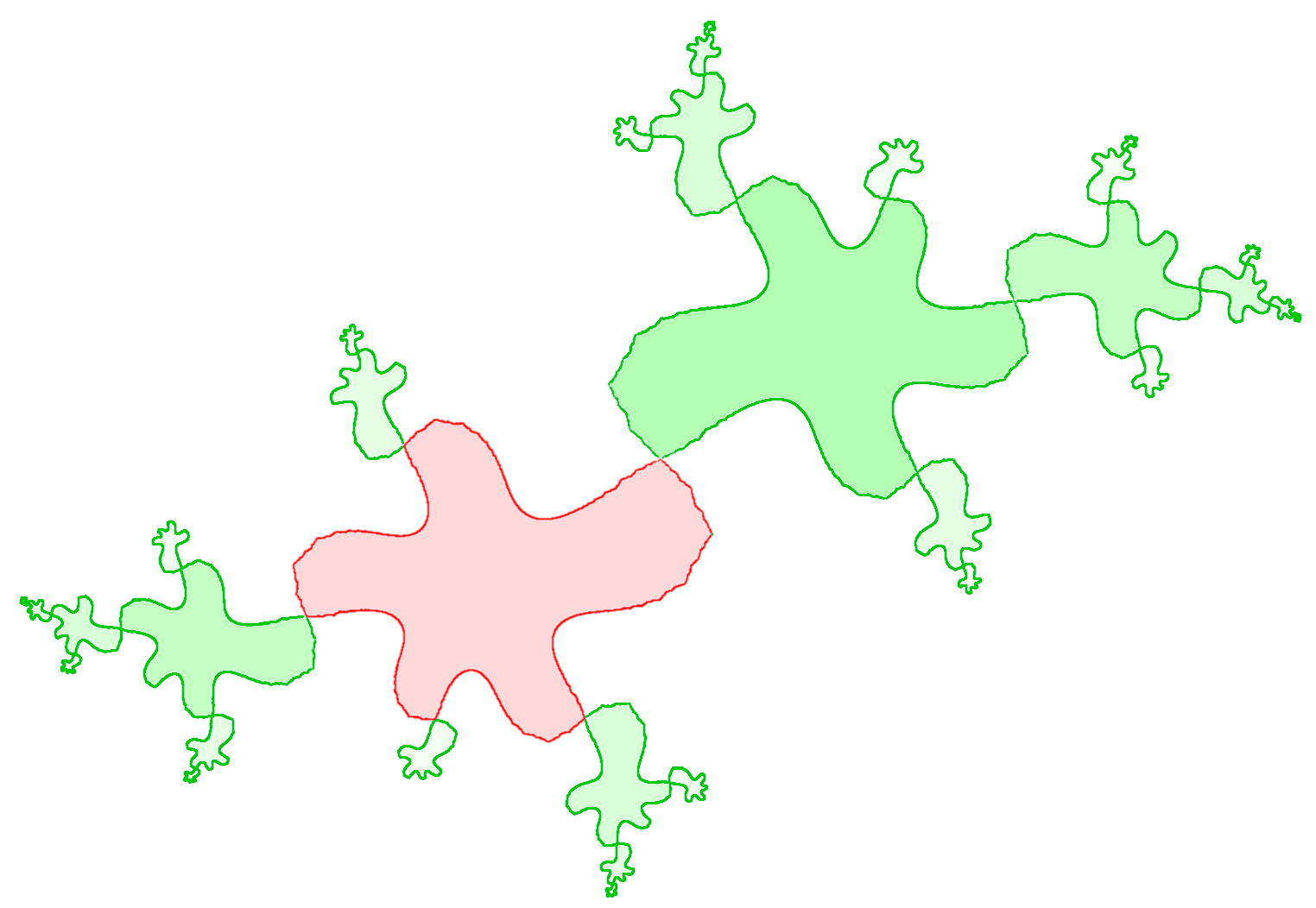}};
    \begin{scope}[
        x={(image.south east)},
        y={(image.north west)}
    ]
    \node [black] at (0.5,0.495) {\scriptsize $\bullet$};
    \node [black] at (0.5,0.46) {\scalebox{0.8}{$c_0$}};
    \node [black] at (0.235,0.32) {\scriptsize$\bullet$};
    \node [black] at (0.255,0.35) {\scalebox{0.8}{$c_{-1}$}};
    \node [black] at (0.443,0.207) {\scriptsize$\bullet$};
    \node [black] at (0.415,0.215) {\scalebox{0.8}{$c_{-2}$}};
    \node [black] at (0.308,0.51) {\scriptsize$\bullet$};
    \node [black] at (0.34,0.505) {\scalebox{0.8}{$c_{-3}$}};
    \node [black] at (0.33,0.208) {\scriptsize $\bullet$};
    \node [black] at (0.29,0.2) {\scalebox{0.8}{$c_{-4}$}};
    \node [red!50!black] at (0.375,0.35) {$Z_{\theta}$};
    \end{scope}
\end{tikzpicture}
    \caption{The Siegel disk $Z_\theta$ in red and the bubbles of generation $\leq 4$ of $f_\theta$ in green where $\theta = [1,1,1,1,100,1,1,\ldots]$}
    \label{fig:bubble-example}
\end{figure}

\subsection{Pseudo-Siegel disks and pseudo-bubbles}

The neutral quadratic polynomial $f_\theta$ admits a unique critical point at 
\[
c_0(f_\theta) := e^{-2\pi i \theta}/2.
\]
We say that an irrational $\theta \in (0,1)$ is of \emph{bounded type} if it has continued fraction expansion $\theta = [0;a_1,a_2,\ldots]$ satisfying $\sup_n a_n < \infty$.
By Douady-Ghys' quasiconformal surgery \cite{D87,G84}, we know that when $\theta$ is of bounded type, $f_\theta$ admits a Siegel disk $Z_\theta$ about $0$, the boundary of $Z_\theta$ is a quasicircle containing $c_0(f_\theta)$, and so the boundary $\partial Z_\theta$ is equal to the postcritical set (the closure of the critical orbit).

It is well-known that the dilatation of the boundary $\partial Z_\theta$ of the Siegel disk worsens as the bound in the continued fraction expansion of $\theta$ increases. 
By employing the near-degenerate analysis of $\partial Z_\theta$, the first author and Lyubich \cite{DL22} recently proved uniform \emph{a priori bounds} for quadratic Siegel disks.

\begin{theorem}[Pseudo-Siegel Bounds \cite{DL22}]
\label{thm:pseudo-siegel-bounds}
    There exists a universal constant $K>1$ such that for every bounded-type irrational $\theta$, there exists a closed, explicitly constructed $K$-quasidisk $\hat{Z}_\theta$ containing $Z_\theta$ such that $f_\theta :\hat{Z}_\theta \to \C$ is injective. 
\end{theorem}

The uniform quasidisk $\hat{Z}_\theta$ is called a \emph{pseudo-Siegel disk} of $f_\theta$. 
It is constructed from $\overline{Z_\theta}$ by filling in parabolic fjords; see Figure~\ref{fig:PS-disk}.
Unlike $\overline{Z_\theta}$, pseudo-Siegel disks $\hat{Z}_\theta$ exist for all neutral maps~\eqref{eq:NeutrFamily}. 

One can pass to the external coordinates $f_\theta \leadsto \fext$ (see~\eqref{eq:dfn:f_ext}) by uniformizing the complement of $\overline{Z_\theta}$.
Pseudo-Siegel bounds yield uniform real \emph{a priori} bounds on the unit circle for $\fext$ (Theorem \ref{thm:R-APB}). 
In this way, external maps $\{\fext\}$ associated with the Neutral Family~\eqref{eq:NeutrFamily} are similar to (uni-)critical circle maps. 
By this analogy, the next step is to understand the \emph{complex dynamics} of $\fext$ beyond the unit circle.

The key players of this paper are bubbles and associated pseudo-bubbles, which are defined as follows.

\begin{definition}
\label{def:bubbles}
    Let $\theta$ be a bounded-type irrational.
    The \emph{bubble} $B_{\theta}^0$ of $f_\theta$ of generation $0$ is the closure of $f_\theta^{-1}(Z_\theta)\backslash Z_\theta$; it is contained in a \emph{pseudo-bubble} $\hat{B}_{\theta}^0$ of generation $0$, which is the closure of the connected component of $f_\theta^{-1}\left( \textnormal{int}\hat{Z}_\theta\right)$ disjoint from $Z_\theta$.
    In general, a \emph{bubble} $B$ of $f_\theta$ of generation $g>0$ is a connected component of $f_\theta^{-g}(B_{\theta}^0)$, and it is contained in a \emph{pseudo-bubble} which is the unique connected component of $f_\theta^{-g}\left(\hat{B}_{\theta}^0 \right)$ containing $B$. 
    The \emph{root} of a (pseudo-)bubble of generation $g\geq 0$ is the unique critical point of $f_\theta^{g+1}$ contained in it.
\end{definition}

See Figure \ref{fig:bubble-example} for an illustration.
In this paper, we prove uniform bounds on the geometry and regularity of all pseudo-bubbles. 
Here we state the results in the original dynamical planes of $f_\theta$, although the proof is carried out for $\fext$.  

\begin{thmx}[Pseudo-Bubble Bounds]
    \label{main-thm-01}
    There exists a constant $K> 1$ such that the following holds for any bounded-type irrational $\theta$ and any bubble $B$ of $f_\theta$. 
    Let $p_n/q_n$ denote the $n$\textsuperscript{th} best rational approximation of $\theta$ and let $c_{m} := (f_\theta|_{\partial Z_\theta} )^m (c_0(f_\theta))$, $m \in \Z$ be the unique critical orbit on the boundary of $Z_\theta$.
    \begin{enumerate}[label=\textnormal{(\arabic*)}]
        \item Shape control: The pseudo-bubble $\hat{B}$ corresponding to $B$ is a $K$-quasidisk.
        \item Suppose $B$ is rooted at some pre-critical point $c_{-k}$ on the boundary of $Z_\theta$. 
        \begin{enumerate}[label=\textnormal{(\alph*)}]
            \item Scale control: The Euclidean diameter of $B$ is determined by the combinatorial scale of its root, that is,
            \[
                K^{-1} |c_{-k} - c_{q_n-k}| \leq \textnormal{diam}(B) \leq \textnormal{diam}\left(\hat{B}\right) \leq K \, |c_{-k} - c_{q_n-k}|,
            \]
            where $n \in \N$ is such that $q_n \leq k < q_{n+1}$.
            \item Angular control: Let $\gamma$ be the unique infinite hyperbolic geodesic in $\C \backslash \overline{Z_\theta}$ landing at $c_{-k}$. Then, every point in $\hat{B}$ is of hyperbolic distance at most $K$ from $\gamma$.
        \end{enumerate}
        \item Size control: If $B$ is not rooted on the boundary of $Z_\theta$, then with respect to the hyperbolic metric of $\C \backslash \overline{Z_\theta}$, the diameter of $\hat{B}$ is bounded above by $K$.
    \end{enumerate}
\end{thmx}

A summary of the proof can be found in \S\ref{ss:summary}.
Theorem \ref{main-thm-01} solves exactly the geometric obstruction that prevents us from proving rigidity and dynamical universality of neutral maps across all irrational rotation numbers; refer to \S\ref{ss:applications} for more details.

\subsection{Outline of the argument}
\label{ss:summary}

Let $\theta$ be a bounded-type irrational.
Denote \[f = f_\theta,\quad Z = Z_\theta,\quad \hat{Z} = \hat{Z}_\theta.\] 

We establish Theorem \ref{main-thm-01} by running the near-degenerate regime in the dynamical plane of the external map $\fext$. This map is obtained by uniformizing $\C \backslash \overline{Z}$, the exterior of the Siegel disk, to $\C \backslash \overline{\D}$, the exterior of the standard unit disk. 
The external map naturally extends to a minimal circle map that has the same rotation number as $f$ and is analytic except at one singular point which corresponds to the critical point. 
Similar to critical circle maps \cite{H87,Sw87}, $\fext$ satisfies \emph{real a priori bounds} (Theorem \ref{thm:R-APB}). 
The transfer from the dynamical plane of $f$ to that of $\fext$ is uniformly quasiconformal outside of the pseudo-Siegel disk $\hat{Z}$ (Corollary \ref{cor:new-psi}), hence we can reduce most of the analysis to the external coordinates.

The idea of using external coordinates dates back to the work of Perez-Marco \cite{PM97} in which he studied the existence of local hedgehogs for neutral holomorphic germs. 
Unlike our setting, the external map in his setting is an analytic circle diffeomorphism with the same rotation number as the map. The external map has appeared previously in the work of Avila and Lyubich \cite{AL22}, in which they formalized a more general notion of circle maps called \emph{quasicritical circle maps}.
The operation $f \leadsto \fext$ can be seen as the inverse of Douady-Ghys surgery.

In~\S\ref{sec:curve-families}, we consider \emph{primary} bubbles $\Bub^{k}$, that is, bubbles that are attached to pre-critical points $\crit_{-k}$ of $\fext$ along the unit circle. 
For $k \geq 0$, $\Bub^{k}$ comes with a natural curve family $\mathcal{F}_{k}$. 
The extremal width $W(\mathcal{F}_{k})$ of $\mathcal{F}_{k}$ encodes the conformal distance between the bubble $B^{k}$ and part of the boundary of $Z_\theta$ outside of a small neighborhood of $c_{-k}$. 

Here is the \emph{central idea}. Assume that $\mathcal{F}_{q_n}$ is sufficiently wide. Applying $\fext^{q_n-q_{n-1}}$, we produce a family that ``should'' consequently overflow both $\mathcal{F}_{q_{n-1}}$ and $\mathcal{F}_{q_{n-2}}$. It would follow that at least one of these two latter families is wider than $\mathcal{F}_{q_n}$. This eventually leads to a contradiction because $\mathcal{F}_{q_{0}}=\mathcal{F}_{1}$ is bounded. Hence, the saying:
\begin{center}
    \textit{``If life is bad today, it was even worse yesterday.''}
\end{center} 

This strategy is realized in Section \ref{sec:curve-families}; the exact statement (Theorem \ref{thm:amplification}) is
\begin{equation}
\label{eq:outline.main}
    W(\mathcal{F}_{q_{n-4}}) \geq \frac{4}{3} W(\mathcal{F}_{q_n}) - O(1) \qquad \text{ for } n \geq 4.
\end{equation}
For a more detailed outline, see Items~\ref{item:outline.1} -- \ref{item:outline.5} after Theorem~\ref{thm:curve-families}.

The implementation of the near-degenerate regime in this paper is organized so as \emph{not to deal with any ``external'' critical point} of $\fext^k$ outside the unit disk. Near the unit disk, these external critical points emerge in large numbers and evolve into essential singularities of the limiting transcendental dynamics; see~\cite{DL23} and~\S\ref{sss:alpha'.Dynam}. The number of involved critical points would affect a near-degenerate threshold, e.g., ``$-O(1)$'' as in~\eqref{eq:outline.main}. This is a primary technical reason for using the iterates $\fext^{q_n-q_{n-1}}$, which form the $\alpha'$-Dynamics. 

Consequently, we do not employ the Covering Lemma directly in this paper; rather, it is employed indirectly as one of the ingredients of~\cite{DL22}.

\subsection{Remarks, applications, and further developments}
\label{ss:applications} A primary consequence of Theorem~\ref{main-thm-01} is that it justifies the bubble behavior illustrated in Figure~\ref{fig:bubble-zoom} for all rotation numbers. Such control is an important ingredient in the bounded-type theory, see~\S\ref{sss:alpha'.Dynam}, but it has been unavailable without the bounded-type assumption, even within the near-parabolic renormalization theory put forward by Inou and Shishikura~\cite{IS}. In~\S\ref{sss:generalization} and~\S\ref{sss:butterfly-rigidity}, we discuss consequences of Theorem~\ref{main-thm-01}, while~\S\ref{sss:hoglet.singul} emphasizes the singular nature of the critical points of $\fext^{q_n}$ on the unit circle: small neighborhood of $c_1$ can \emph{blow-up} under $\fext^{-q_n}$.

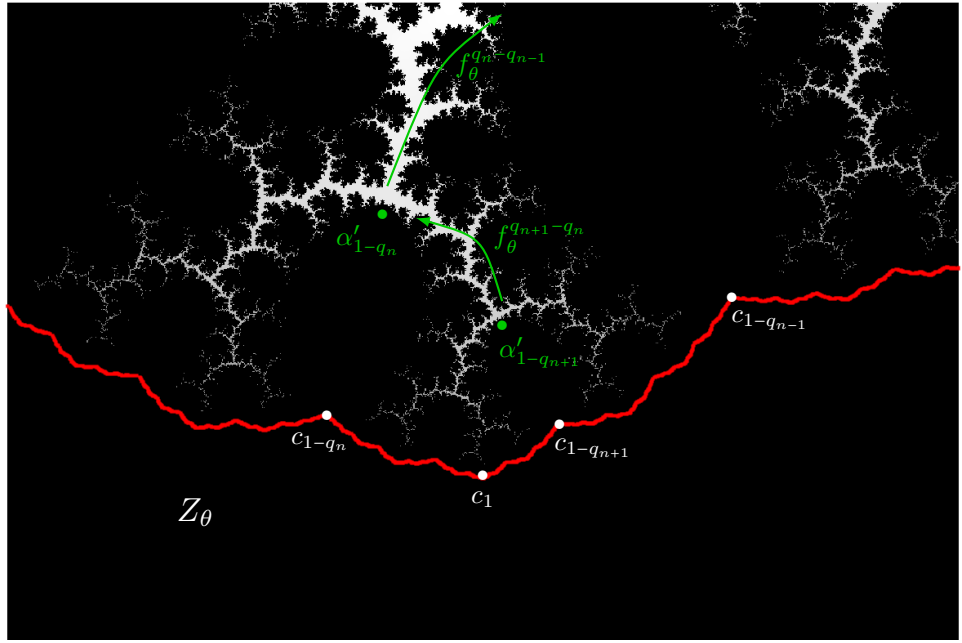
\begin{figure}
    \centering
\begin{tikzpicture}
    \node[anchor=south west,inner sep=0] (image) at (0,0) {\includegraphics[width=1\linewidth]{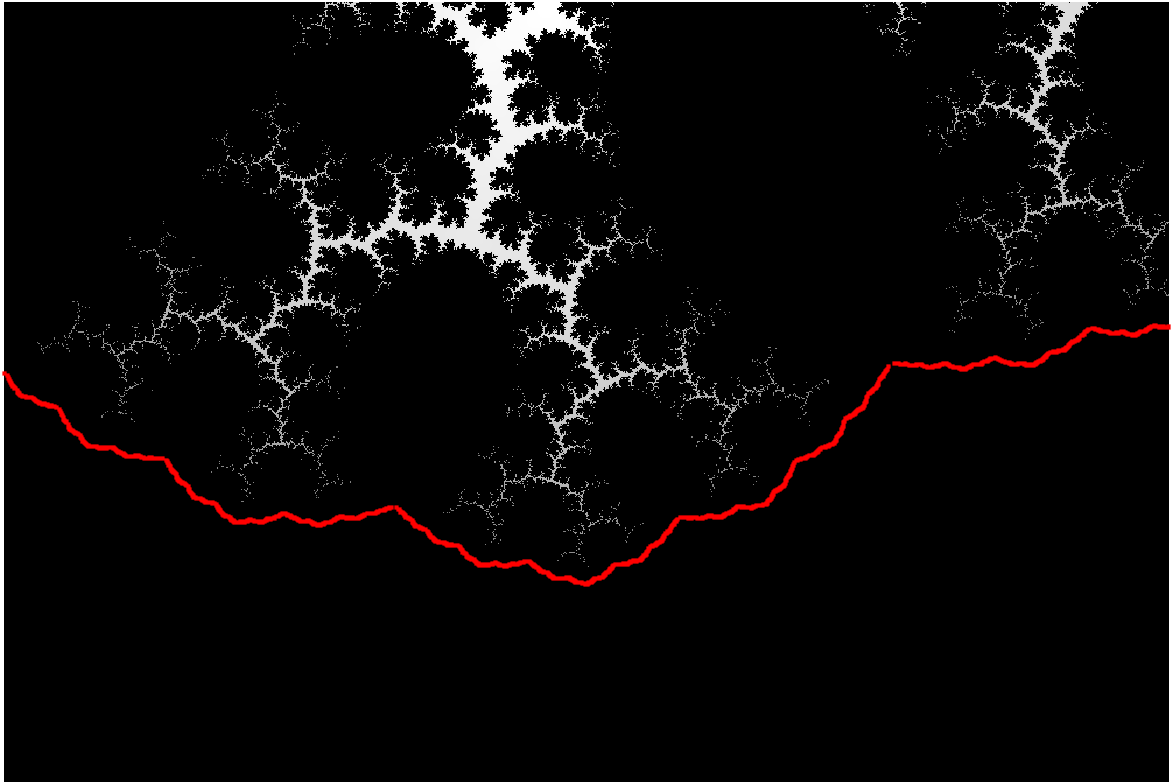}};
    \begin{scope}[
        x={(image.south east)},
        y={(image.north west)}
    ]
    \node [white] at (0.2,0.2) {\scalebox{1.25}{$Z_{\theta}$}};
    
    \node [white] at (0.5,0.255) {$\bullet$};
    \node [white] at (0.5,0.22) {$c_{1}$};
    \node [white] at (0.337,0.35) {$\bullet$};
    \node [white] at (0.33,0.31) {$c_{1-q_n}$};
    \node [white] at (0.58,0.335) {$\bullet$};
    \node [white] at (0.615,0.3) {$c_{1-q_{n+1}}$};
    \node [white] at (0.76,0.534) {$\bullet$};
    \node [white] at (0.8,0.5) {$c_{1-q_{n-1}}$};

    \node [green!80!black] at (0.52,0.49) {$\bullet$};
    \node [green!80!black] at (0.56,0.45) {$\alpha'_{1-q_{n+1}}$};
    \node [green!80!black] at (0.395,0.665) {$\bullet$};
    \node [green!80!black] at (0.38,0.625) {$\alpha'_{1-q_{n}}$};
    
    \draw [green!80!black,thick,-latex] (0.52,0.53) .. controls (0.5,0.63) .. (0.43,0.66);
    \node [green!80!black] at (0.56,0.63) {$f_\theta^{q_{n+1}-q_n}$};
    \draw [green!80!black,thick,-latex] (0.4,0.71) .. controls (0.45,0.9) .. (0.52,0.98); 
    \node [green!80!black] at (0.52,0.9) {$f_\theta^{q_{n}-q_{n-1}}$};
    \end{scope}
\end{tikzpicture}
    \caption{The dynamical plane of $f_\theta$ when $\theta = [2,1,1,1,\ldots]$ near the critical value $c_1=c_1(f_\theta)$. 
    The boundary of the Siegel disk $Z_\theta$ is colored red.
    The bubbles near $c_1$, shown in black, are uniformly bounded relative to their distance to $c_1$.
    The associated $\alpha'$-Dynamics (see~\S\ref{sss:alpha'.Dynam}) are illustrated in green.}
    \label{fig:bubble-zoom}
\end{figure}

\subsubsection{Bubbles and the $\alpha'$-Dynamics in the bounded-type case}\label{sss:alpha'.Dynam}
The control of bubbles and the $\alpha'$-points are an important part of the \emph{bounded-type} Siegel renormalization theory. It is shown in~\cite{McM98} that every point of $\partial Z_\theta$ is a measurable deep point of the filled Julia set of $f_\theta$, i.e. bubbles occupy ``exponentially fast'' most of the Lebesgue area of $\C\setminus \overline Z_\theta$ near $\partial Z_\theta$, cf, Figure~\ref{fig:bubble-zoom}. This leads to the dynamical universality of $f_\theta$ (see~\cite[Theorem 8.1]{McM98}) where the $f^{q_n}$-dynamics near $\partial Z_\theta$ converge exponentially fast toward a limiting transcendental dynamics.

Every bubble $B_i$ of $f_\theta$ has a center $\alpha'_i\in B_i$; it is the unique iterated preimage of the $\alpha$-fixed point of $\overline Z_\theta$ within $B_i$. In deep scale, the $\alpha'_i$ points evolve into the roofs of the associated transcendental bubbles. It is shown in~\cite[Section 5]{DL23} that the limiting $\alpha'$-points are exactly the branched points of the transcendental escaping sets consisting of essential singularities of the associated cascades. These points form the hyperbolic set of the limiting $\alpha'$-Dynamics induced by $f_{\theta}^{q_n-q_{n-1}}$. This leads to the puzzle theory~\cite{DL23} of the transcendental dynamics arising on the unstable manifolds of the respective renormalizations.

\subsubsection{Hoglets and pseudo-bubbles for all rotation numbers}
\label{sss:generalization}

One important application of pseudo-Siegel bounds is the existence of a \emph{Mother Hedgehog} $H_\theta$, a full planar continuum that contains the neutral fixed point $0$ and the critical point $c_0(f_\theta)$ such that $f_\theta : H_\theta \to H_\theta$ is a homeomorphism.
Indeed, by the compactness of uniform quasidisks, for every irrational $\theta$, $f_\theta$ admits a uniform pseudo-Siegel disk $\hat{Z}_\theta$. Consequently, the Hausdorff limits of bounded-type Siegel disks construct Mother Hedgehogs $H_\theta$ for all $\theta$.

 Theorem~\ref{main-thm-01} implies that, for every $\theta$, the iterated lifts of the Mother Hedgehog $H_\theta$ and its enclosing pseudo-Siegel disk $\hat{Z}^{-1}_\theta\supset H_\theta$ are arranged as predicted by the arithmetic of $\theta$. Iterated preimages of $\hat{Z}_{\theta}$ are still called \emph{pseudo-bubbles} (as they are locally connected) but iterated preimages of $H_\theta$ are called \emph{hoglets}. Thus every hoglet comes with a uniformly qc pseudo-bubble enclosing it. This is an important ingredient for the Pullback Argument with respect to hoglet butterflies; see~\S\ref{sss:butterfly-rigidity}.

 In \cite{DL26}, it is further shown that the pair $\hat{Z}_\theta\supset H_\theta$  can be constructed to depend continuously on $\theta$ with respect to the ``parabolic compactification'' $\overline{\Theta}$ on the space of irrationals. Theorem~\ref{main-thm-01} can be used to ``appropriately extend'' such continuous dependence to hoglet and enclosing pseudo-bubbles, cf, Figure~\ref{fig:bubble-zoom}.

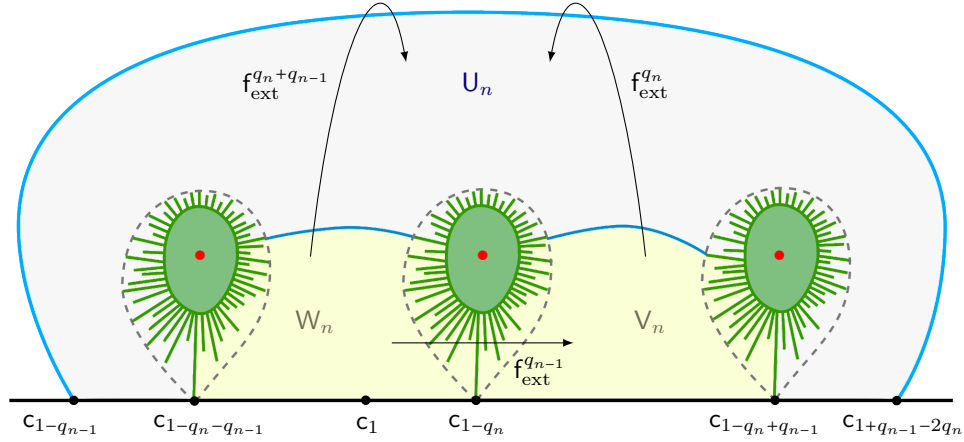
\begin{figure}
    \centering
\begin{tikzpicture}
    \node[anchor=south west,inner sep=0] (image) at (0,0) {\includegraphics[width=1\linewidth]{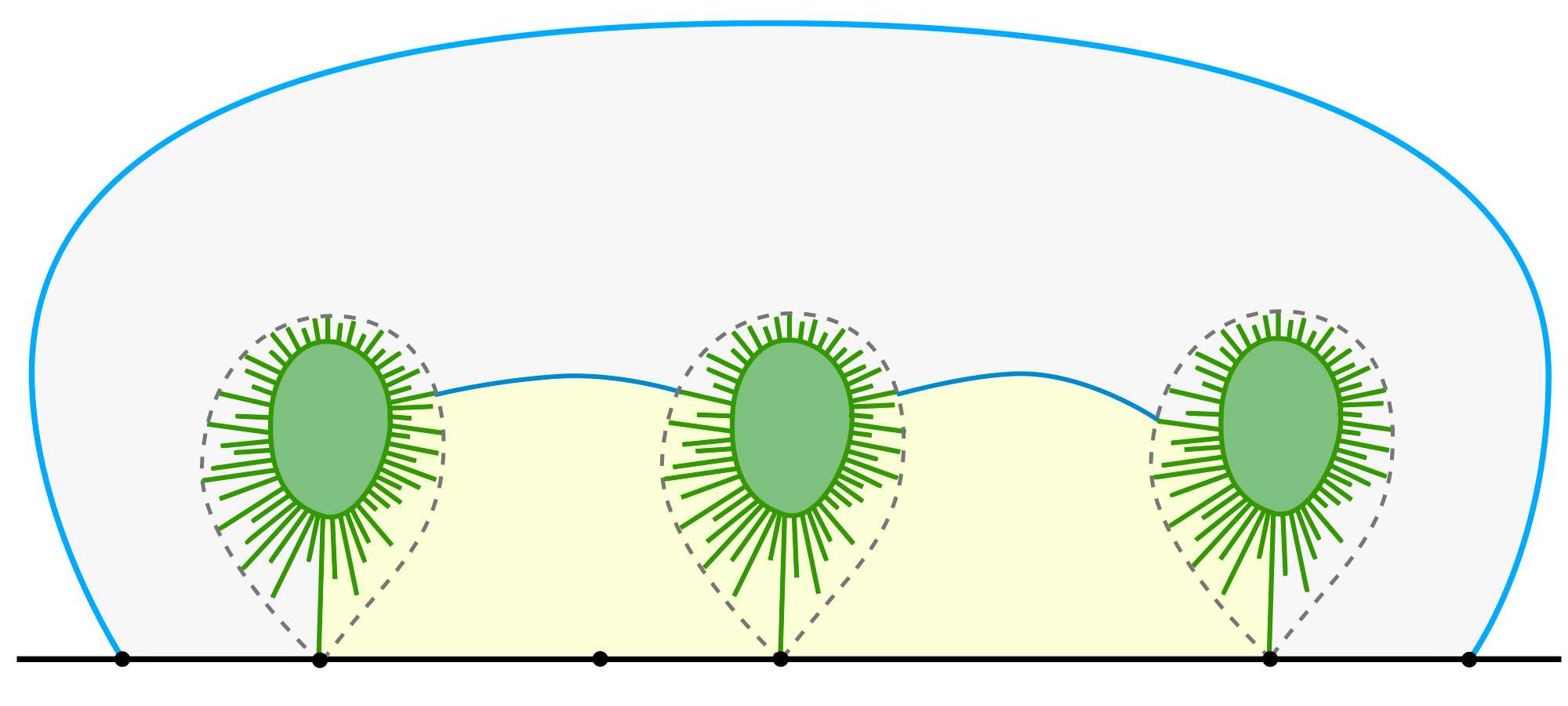}};
    \begin{scope}[
        x={(image.south east)},
        y={(image.north west)}
    ]
    \node [blue!50!black] at (0.5,0.8) {$\mathsf{U}_n$};
    \node [yellow!30!black] at (0.68,0.25) {$\mathsf{V}_n$};
    \node [yellow!30!black] at (0.33,0.25) {$\mathsf{W}_n$};
    \node [black] at (0.065,0.01) {$\crit_{1-q_{n-1}}$};
    \node [black] at (0.22,0.01) {$\crit_{1-q_n-q_{n-1}}$};
    \node [black] at (0.385,0.01) {$\crit_{1}$};
    \node [black] at (0.5,0.01) {$\crit_{1-q_{n}}$};
    \node [black] at (0.80,0.01) {$\crit_{1-q_n+q_{n-1}}$};
    \node [black] at (0.945,0.01) {$\crit_{1+q_{n-1}-2q_n}$};
    \draw[-latex] (0.675,0.4) .. controls (0.65,0.9) and (0.6,1.15) .. (0.575,0.85);
    \node [black] at (0.68,0.8) {$\fext^{q_n}$};
    \draw[-latex] (0.325,0.4) .. controls (0.35,0.9) and (0.4,1.15) .. (0.425,0.85);
    \node [black] at (0.3,0.8) {$\fext^{q_n+q_{n-1}}$};
    \draw[-latex] (0.41,0.2) -- (0.6,0.2);
    \node [black] at (0.565,0.135) {$\fext^{q_{n-1}}$};
    
    \node [red] at (0.21,0.4) {$\bullet$};
    \node [red] at (0.505,0.4) {$\bullet$};
    \node [red] at (0.815,0.4) {$\bullet$};
    \end{scope}
\end{tikzpicture}
    \caption{An illustration of the uniform hoglet butterfly structure for sufficiently high level $n$; see also~\S\ref{sss:butterfly-rigidity}. 
    The maps $\fext^{q_n+q_{n-1}}: \mathtt{W}_n \to \mathtt{U}_n$ and $\fext^{q_{n}}:\mathtt{V}_n \to \mathtt{U}_n$ are univalent.
    The outer blue arc bounding $\mathtt{U}_n$ has diameter $\asymp |\crit_1-\crit_{-q_n}|$, meets the real line at uniform angles, and is uniformly separated from the union of $\mathtt{W}_n$, $\mathtt{V}_n$, and the three pseudo-bubbles protecting the hoglets in green.
    }
    \label{fig:butterfly}
\end{figure}

\subsubsection{Hoglet butterfly bounds and rigidity}
\label{sss:butterfly-rigidity} Pseudo-Siegel bounds from~\cite{DL22} imply (by~\cite{DL26}) the existence of the renorm-attractor for the whole Neutral Family~\eqref{eq:NeutrFamily} under the sector renormalization. In a follow-up work \cite{DLL} with Misha Lyubich, we prove the rigidity of this renorm-attractor: any two combinatorially equivalent bi-infinite neutral towers in the renorm-attractor are conformally conjugate. A key geometric construction in \cite{DLL} is the uniform \emph{hoglet butterfly} structure illustrated in Figure \ref{fig:butterfly}. This structure is only made possible thanks to Theorem \ref{main-thm-01} of this paper.

Butterfly structures have been established in the theory of critical circle maps \cite{dF99,Y99,dFdM2}. They serve as a preparation for the Pullback Argument. For bounded-type $\theta$, butterfly structures for the external map $\fext$ have been justified in \cite{AL22} using the framework of quasicritical circle maps. For general $\theta$, however, two fundamental differences arise compared with the above-mentioned constructions. Both arise from the non-JLC phenomenon in the Neutral Family~\eqref{eq:NeutrFamily}: as Figure~\ref{fig:butterfly} illustrates, the hoglets that form the butterflies need not be locally connected.

Firstly, $\fext$ can be discontinuous at the relevant critical points; see~\S\ref{sss:hoglet.singul} and Figure \ref{fig:singularity}. Secondly, in the (near-)Cremer case, the local geometry near $\alpha'_{1-q_n}$ encodes the anti-renormalization history. We therefore do not expect a hybrid conjugacy between ``just'' forward-infinite combinatorially equivalent towers. The Pullback Argument in \cite{DLL} is carried out for \emph{bi-infinite} towers, using uniformly qc pseudo-bubbles (see~\S\ref{sss:generalization}) around hoglets as a preparatory ingredient.

We believe that rigidity of bi-infinite towers suffices to establish the uniform hyperbolicity of the full neutral renormalization; a strategy is outlined in~\cite{DLL}. 

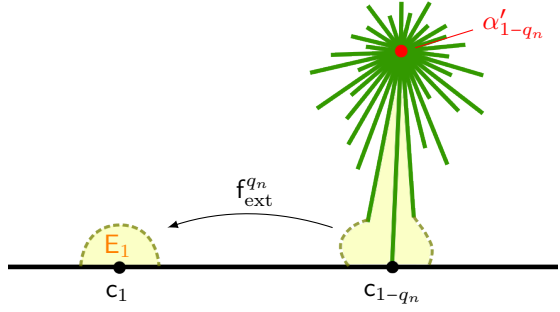
\begin{figure}
    \centering
\begin{tikzpicture}
    \node[anchor=south west,inner sep=0] (image) at (0,0) {\includegraphics[width=0.575\linewidth]{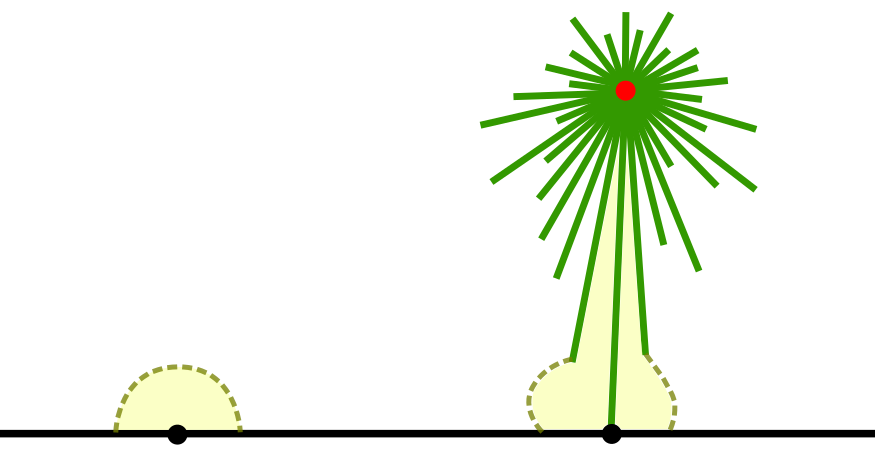}};
    \begin{scope}[
        x={(image.south east)},
        y={(image.north west)}
    ]
    \node [black] at (0.2,0) {$\crit_{1}$};
    \node [orange] at (0.2,0.16) {${\mathsf E}_{1}$};
    \node [black] at (0.7,0) {$\crit_{1-q_{n}}$};
    \draw[-latex] (0.59,0.22) .. controls (0.49,0.28) and (0.39,0.28) .. (0.29,0.22);
    \node [black] at (0.45,0.34) {$\fext^{q_{n}}$};
    \node [red] at (0.92,0.9) {$\alpha'_{1-q_n}$};
    \draw [red] (0.845,0.88) -- (0.74,0.82);
    \end{scope}
\end{tikzpicture}
    \caption{Hoglet-type singularity for $\fext^{q_n}$ in the Cremer case: a small neighborhood ${\mathsf E}_{1}\subset \C\setminus \overline {\mathbb D}$ of $\crit_1$ blows-up under $\fext^{q_n}$, see~\eqref{eq:E1.blowup}.}
    \label{fig:singularity}
\end{figure}

\subsubsection{Hoglet-type singularities at critical points of $\fext^{q_n}$.}\label{sss:hoglet.singul} 
One fundamental difference between the external map $\fext$ and critical circle maps is that the lift of points near the critical value blows up as shown in Figure \ref{fig:singularity}.

Let $\mathsf{E}_1$ be any neighborhood of $\crit_1$ in $\C \backslash \overline{\D}$ with Euclidean diameter much smaller than $|\crit_1 - \crit_{1-q_n}|$.
There are exactly two connected components $\mathsf{E}^l_{1-q_n}, \mathsf{E}^r_{1-q_n}$ of $\fext^{-q_n}(\mathsf{E}_1)$ attached to the pre-critical point $\crit_{1-q_n}$. 
In the Cremer case, the domain $\mathsf{E}^{l/r}_{1-q_n}$ reaches the point $\alpha'_{1-q_n}$, the center of the hoglet attached to $\crit_{1-q_n}$.
Consequently,
\begin{equation}
\label{eq:E1.blowup}
\diam \big(\mathsf{E}^{l/r}_{1-q_n} \big)\ \ge |\crit_{1-q_n} - \alpha'_{1-q_n}| \succeq\ |\crit_1 - \crit_{1-q_n}|\ \gg \diam (\mathsf{E}_1).    
\end{equation}

This observation implies that, without Theorem~\ref{main-thm-01}, the diameter of the butterfly wings $\mathtt{W}_n$ and $\mathtt{V}_n$ in Figure \ref{fig:butterfly} could become arbitrarily large relative to $|\crit_1 - \crit_{1-q_n}|$, thereby preventing the uniform butterfly structure. To the best of our understanding, this possibility is not ruled out by the nonuniform hyperbolic contraction of $\fext^{-1}$ outside the unit disk.

\subsection{Acknowledgements}

We would like to thank Misha Lyubich for various discussions on neutral renormalization.
The second author is grateful to the Institute for Mathematical
Sciences at Stony Brook for their hospitality. The first author was partially supported by the Simons Fellows grant and by the NSF grants DMS 2055532 and DMS 2246485.


\section{Pseudo-Siegel bounds}

In this section, we summarize the main results of \cite{DL22} and discuss some of the details that will be necessary for later sections.

\subsection{Pseudo-Siegel disks} 
\label{ss:pseudo-siegel}

Denote the set of irrationals by $\Theta := (0,1) \backslash \Q$. For $\theta \in \Theta$, we write the continued fraction expansion as
\[
    \theta = [0;a_1,a_2,a_3,\ldots]:= \cfrac{1}{a_1 + \cfrac{1}{a_2 + \frac{1}{a_3 + \ldots }}}.
\]
For $n \geq 1$, the $n$\textsuperscript{th} best rational approximation of $\theta$ is the rational number $p_n/q_n:= [0;a_1,\ldots,a_n]$. 
We say that $\theta$ is 
\begin{itemize}
    \item of \emph{bounded type} if $\sup_n a_n < \infty$, and 
    \item \emph{eventually golden mean} (EGM) if $\lim_{n\to \infty} a_n = 1$. 
\end{itemize}
Denote by $\Theta_{\text{bdd}}$ the set of bounded-type irrationals, and by $\Theta_{\text{\text{EGM}}}$ the set of EGM irrationals.

In what follows, we will usually choose a fixed irrational $\theta$ in either $\Theta_{\text{bdd}}$ or $\Theta_{\text{EGM}}$.
We say that a constant is \emph{universal} if it is independent of the choice of $\theta$.

For $\theta \in \Theta_{\text{bdd}}$, it is known from Douady-Ghys surgery \cite{D87,G84} that the quadratic polynomial $f_\theta$ admits a Siegel disk $Z_\theta$ with quasicircular boundary passing through the critical point.
In particular, the Riemann mapping $(Z_\theta,0) \to (\D,0)$ extends to a global quasiconformal map and conjugates $f_\theta|_{\partial Z_\theta}$ and the rigid rotation by angle $\theta$ on $\partial \D$.
We equip $\partial Z_\theta$ with the \emph{combinatorial metric} that is the unique invariant normalized metric.

We say that an interval $I$ on $\partial Z_\theta$ is a \emph{combinatorial interval of level $n$} if it is the shortest interval connecting a point $x \in \partial Z_\theta$ and $f^{q_n}_\theta(x)$, where $q_n$ is the denominator of the $n$\textsuperscript{th} best rational approximation of $\theta$.
All combinatorial intervals of level $n$ have the same combinatorial length, which we will denote by $l_n$. These $l_n$'s are related by the equation
\[
    l_{n-1} = a_{n+1} l_n + l_{n+1}.
\]

Given a combinatorial interval $I \subset \partial Z_\theta$ of some level $n$, we denote by $W^+_3(I)$ the extremal width of the family of curves in $\RS \backslash \overline{Z_\theta}$ that start from $I$ and end at $\partial Z_\theta \backslash (L \cup I \cup R)$, where $L$ and $R$ are the neighboring level $n$ combinatorial intervals on the left and right of $I$ respectively. 
A summary of the key properties of extremal width can be found in Appendix \ref{sec:appendix}.

\begin{theorem}[Uniform External Bounds \cite{DL22}]
\label{thm:uniform-external}
    There exists a universal $K>0$ such that for every $\theta \in \Theta_{\textnormal{bdd}}$ and every combinatorial interval $I$ on the boundary of the Siegel disk $Z_\theta$ of $f_\theta$, we have $W^+_3(I) \leq K$.
\end{theorem}

The proof of this theorem relies on Theorem \ref{thm:pseudo-siegel-bounds}, in particular the construction of \emph{$K$-pseudo-Siegel disk} $\hat{Z}_\theta$ for $\theta \in \Theta_{\textnormal{bdd}}$.
To be more precise, an infinite nest of pseudo-Siegel disks 
\[
    \ldots \subset \hat{Z}^2_\theta \subset \hat{Z}^1_\theta \subset \hat{Z}^0_\theta \subset \hat{Z}^{-1}_\theta = \hat{Z}_\theta
\]
is constructed explicitly. 
Below, we will summarize the key ideas of the construction. 

\begin{figure}
    \centering
\begin{tikzpicture}
    \node[anchor=south west,inner sep=0] (image) at (0,0) {\includegraphics[width=1\linewidth]{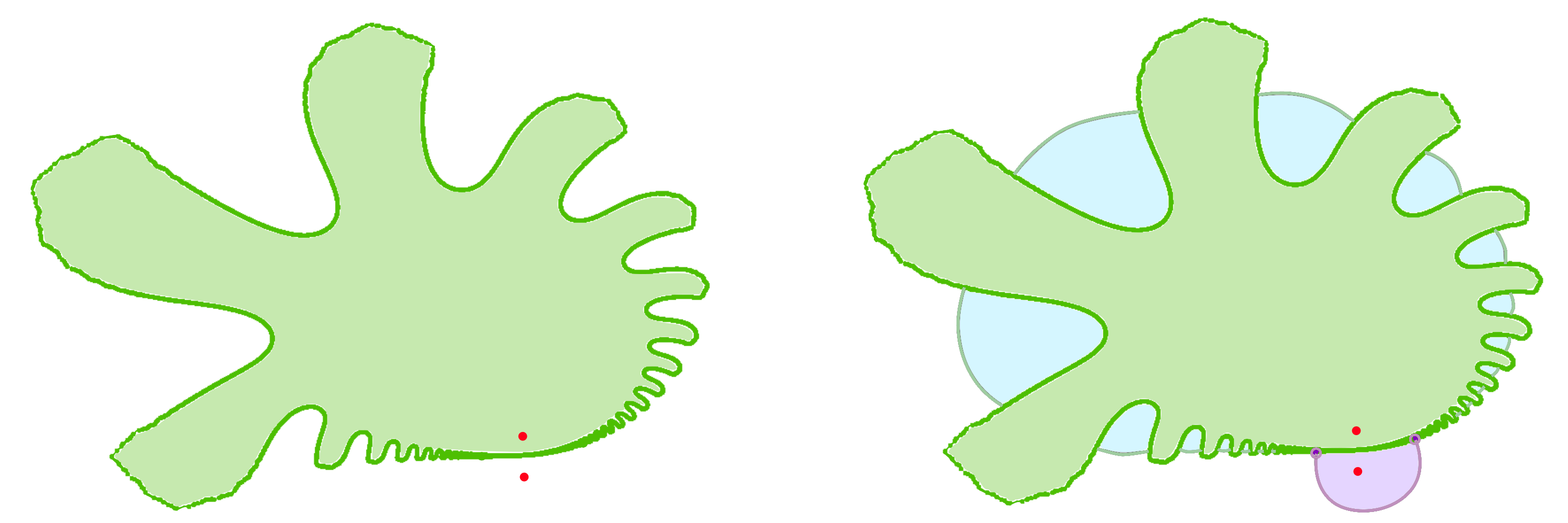}};
    \begin{scope}[
        x={(image.south east)},
        y={(image.north west)}
    ]
    
    \node [green!20!black] at (0.3,0.4) {$Z$};
    \node [green!20!black] at (0.83,0.4) {$\hat{Z}$};
    \end{scope}
\end{tikzpicture}
    \caption{The Siegel disk and a pseudo-Siegel disk of $f_\theta$ when $\theta = [0;100,100,2,1,1,\ldots]$. Level $-1$ dam bounds the parabolic fjord in purple and level $0$ dams bound the fjords in blue. The alpha and beta fixed points are marked in red.}
    \label{fig:PS-disk}
\end{figure}

To simplify our exposition, we will assume that $a_1 \geq 2$ or equivalently $\theta<1/2$; if otherwise, then we can define $\hat{Z}^m_\theta$ as the reflection of $\hat{Z}^m_{1-\theta}$. 
We will also denote $f = f_\theta$, $Z = Z_\theta$, and $\hat{Z}^m = \hat{Z}^m_\theta$. 
For $k \geq 1$, denote by $\text{CP}(f^k)$ the set of critical points of $f^k$.
 
For every $n \geq -1$, denote by $\mathfrak{D}_n$ the level $n$ tiling of $\partial Z$, in which it consists of the closure of every component of $\partial Z \backslash \text{CP}(f^{q_{n+1}})$. 
For $n =-1$, $\mathfrak{D}_{-1}$ consists of a single tile which is the whole boundary of $Z$.
For $n \geq 0$, every tile in $\mathfrak{D}_n$ has combinatorial length either $l_n$ or $l_n + l_{n+1}$.

Assume first that $\theta$ is EGM. For all sufficiently high $m \in \N$, $\hat{Z}^m$ is set to be equal to the closure of the Siegel disk $Z$ of $f$. 
In general, for all $m\geq -1$, $\hat{Z}^m$ is constructed inductively as follows. See Figure \ref{fig:PS-disk}.

Let us fix a large constant $\Mbold \in \N$ and set $\Mbold' = \lfloor e^{\sqrt{\log \Mbold}} \rfloor$. There are two cases to consider.
\begin{itemize}
    \item \underline{Bounded regime:} $a_{m+2}< \Mbold$ \\
    We set $\hat{Z}^m$ to be equal to $\hat{Z}^{m+1}$.
    \item \underline{Near-parabolic regime:} $a_{m+2} \geq \Mbold$ \\
    We add (\emph{level} $m$) \emph{dams} $\beta_I$ for each tile $I=[a,b]$ in $\mathfrak{D}_m$. 
    Each dam $\beta_I$ is a hyperbolic geodesic of $\C \backslash \overline{Z}$ connecting two points $a'$ and $b'$. The point $a'$ (resp. $b'$) is the unique point in $\text{CP}(f^{q_{m+2}}) \cap I$ such that the combinatorial distance between $a'$ and $a$ (resp. $b'$ and $b$) is either $\Mbold'l_{m+1}$ or $\Mbold'l_{m+1}+l_{m+2}$. 
    Then, we set $\hat{Z}^m$ to be the unique smallest closed disk containing $\hat{Z}^{m+1}$ and $\beta_I$ for all $I \in \mathfrak{D}_m$.
\end{itemize}

Each of $\hat{Z}^m$ is no longer invariant in general, but it is almost invariant under $f^{q_{m+1}}$ in the following sense. 
The choice of $\Mbold'$ ensures that there exists a collection of pairwise disjoint \emph{collars} $\{A_I\}_{I \in \mathfrak{D}_m}$ such that 
\begin{itemize}
    \item each $A_I$ is an annulus with modulus greater than some definite constant $\boldsymbol{\delta} = \boldsymbol{\delta}(\Mbold)>0$,
    \item $\beta_I \cup f^{q_{m+1}}(\beta_I)$ is contained in the bounded connected component of $\C \backslash A_I$, and
    \item the intersection of $A_I$ and the boundary of the Siegel disk $Z$ is a union of four intervals contained in the interior of $I$. 
\end{itemize}
See Figure \ref{fig:fjord}. 
The third item implies that collars are disjoint from the critical value.
This collar property will be used in \S\ref{ss:uniform-size}.

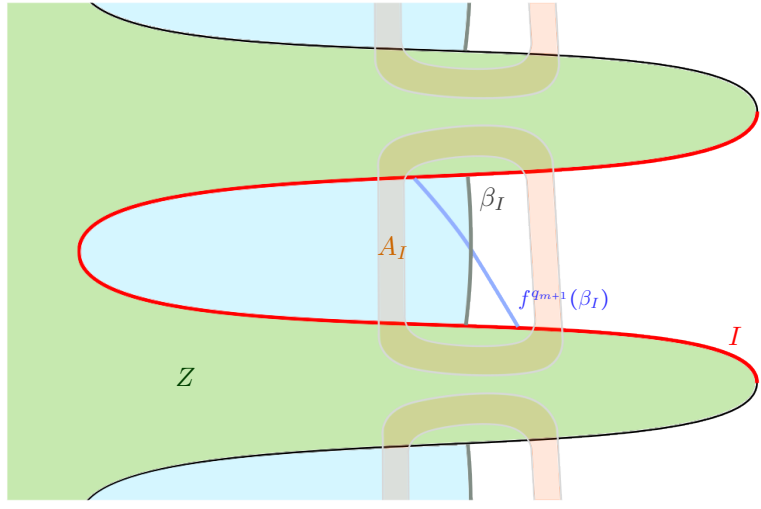
\begin{figure}
    \centering
\begin{tikzpicture}
    \node[anchor=south west,inner sep=0] (image) at (0,0) {\includegraphics[width=0.85\linewidth]{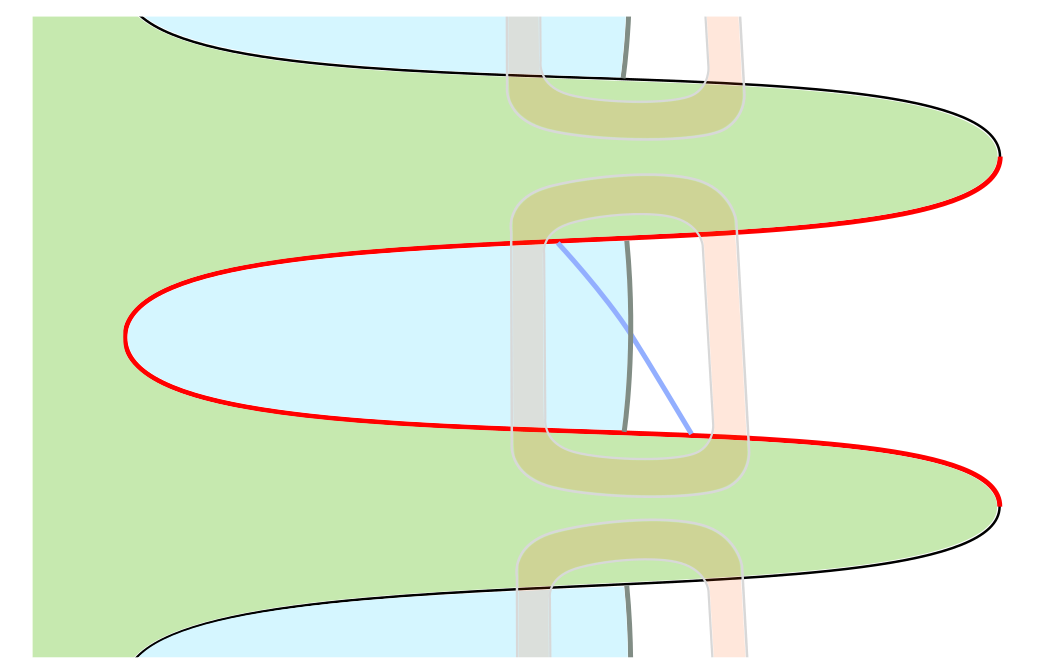}};
    \begin{scope}[
        x={(image.south east)},
        y={(image.north west)}
    ]
    
    \node [red] at (0.925,0.33) {$I$};
    \node [gray!50!black] at (0.628,0.595) {$\beta_I$};
    \node [blue!80!white] at (0.715,0.4) {\scalebox{0.8}{$f^{q_{m+1}}(\beta_I)$}};
    \node [orange!80!black] at (0.505,0.5) {$A_I$};
    \node [green!25!black] at (0.25,0.25) {$Z$};
    \end{scope}
\end{tikzpicture}
    \caption{For $I \in \mathfrak{D}_m$, the collar $A_I$ surrounds the dam $\beta_I$ and its image $f^{q_{m+1}}(\beta_I)$.}
    \label{fig:fjord}
\end{figure}

In the final step of the proof of Theorem \ref{thm:pseudo-siegel-bounds} in \cite[\S11]{DL22}, a nest of tilings $\mathcal{T}(\hat{Z}) = \left\{\mathcal{T}_m(\hat{Z}) \right\}_{m\geq -1}$ of ``essentially bounded geometry'' on $\partial \hat{Z}$ is introduced. 
There exists a constant $K' = K'(\Mbold)>1$ such that $\mathcal{T}_m(\hat{Z})$ satisfies the following properties.
\begin{itemize}
    \item Bounded combinatorics: There exists some universal constant $M \in \N$ such that for all $m \geq -1$, every tile in $\mathcal{T}_m(\hat{Z})$ is a union of at most $M$ tiles of $\mathcal{T}_{m+1}(\hat{Z})$.
    \item Bounded inner geometry: For every tile $I$ in $\mathcal{T}_m(\hat{Z})$, denote by $\mathcal{F}^-_3(I)$ the family of proper curves in the interior of $\hat{Z}$ that start from $I$ and land on $\partial \hat{Z} \backslash (L_I \cup I \cup R_I)$, where $L_I$ and $R_I$ are the two neighboring tiles in $\mathcal{T}_m(\hat{Z})$. 
    Then, the width $W^-_3(I)$ of $\mathcal{F}^-_3(I)$ satisfies
    \[
        W^-_3(I) \leq K'.
    \]
    \item Bounded outer geometry: For every tile $I$ in $\mathcal{T}_m(\hat{Z})$, denote by $\mathcal{F}^+_3(I)$ the family of proper curves in $\RS \backslash \hat{Z}$ that start from $I$ and land on $\partial \hat{Z} \backslash (L_I \cup I \cup R_I)$. 
    Then, the width $W^+_3(I)$ of $\mathcal{F}^+_3(I)$ satisfies
    \[
        W^+_3(I) \leq K'.
    \]
\end{itemize}
This nest of tilings is constructed so that it is compatible with $\mathfrak{D}_m$ in the following sense.
For every $m \geq -1$, we say that an interval in $\mathfrak{D}_m$ is \emph{regular} if both of its endpoints are on the boundary of $\hat{Z}$. For every regular interval $I = [a,b]$ in $\mathfrak{D}_m$, denote by $I'$ the smallest closed subinterval of $\partial \hat{Z}$ with endpoints $a$ and $b$. Then, the regularization $I'$ of $I$ is a tile in $\mathcal{T}_{m}(\hat{Z})$ and if $a_{m+2}\geq \Mbold$, then the dam $\beta_I$ is also a tile in $\mathcal{T}_{m+1}(\hat{Z})$.

Lastly, for any bounded-type non-EGM $\theta$, we can define the level $m$ pseudo-Siegel disk $\hat{Z}^m_\theta$ of $f_\theta$ to be the limit of $\hat{Z}^m_{\theta_n}$ where $\theta_n$ is a sequence of EGM irrationals converging to $\theta$.
This limiting procedure relies on the continuity of bounded-type Siegel disks which is a consequence of the Douady-Ghys surgery.

\begin{remark}
\label{rem:combinatorial-threshold}
    It is important that the constants $\boldsymbol{\delta}$, $K'$, and ultimately $K$ are independent of $\theta$ and $m$. 
    In principle, for every choice of the combinatorial threshold $\Mbold$, we obtain a $K(\Mbold)$-pseudo-Siegel disk $\hat{Z} = \hat{Z}(\Mbold)$ and it has the property that $\hat{Z}(\Mbold+1) \subset \hat{Z}(\Mbold)$.
    As $\Mbold$ increases, the constants $\boldsymbol{\delta}$, $K'$, $K$ will become worse. 
    In this paper, the only moment $\Mbold$ is allowed to vary is in Lemma \ref{lem:disjointness}.
    Aside from that, the combinatorial threshold $\Mbold$ will be treated as a fixed sufficiently large universal integer.
\end{remark}

\begin{remark}
    What we have described above is the \emph{geodesic} pseudo-Siegel disk in which all dams are hyperbolic geodesics.
    In \cite{DLL}, we will consider a more dynamical variant in which for every $m$, the dam for the tile in $\mathfrak{D}_m$ closest to the critical value is geodesic and its preimages constitute the other dams of the same level.
    This dynamical pseudo-Siegel disk still satisfies the main theorems in \cite{DL22} as well as this paper.
\end{remark}

\subsection{Mother Hedgehogs}
\label{ss:mother-hedgehog}

This subsection does not affect Sections \ref{sec:external}--\ref{sec:pseudo-bubble-bounds}.
It solely serves to discuss the extension to all rotation numbers.

\begin{definition}
    For $\theta \in \Theta$, a subset $H$ of $\C$ is called a \emph{hedgehog} of $f_\theta$ if it is a full compact connected subset of $\C$ containing the neutral fixed point $0$ such that $f_\theta: H \to H$ is a homeomorphism. 
    If additionally $H$ contains the critical point, then we call $H_\theta := H$ a \emph{Mother Hedgehog}.
\end{definition}

The existence of a Mother Hedgehog $H_\theta$ for any irrational $\theta$ is a corollary of Theorem \ref{thm:pseudo-siegel-bounds}; it is simply a Hausdorff limit of the closure of bounded-type Siegel disks.
In the unbounded case, $H_\theta$ can have a wild topology.

We identify the space of irrationals $\Theta$ with the sequence space
\[
    \big( \{-,+\} \times \{2,3,4,\ldots \} \big)^{\N}
\]
and consider its natural compactification
\[
    \overline{\Theta} := \big( \{-,+\} \times \{2,3,4,\ldots, \infty\} \big)^{\N}.
\]
Let us equip $\Theta$ with the subspace topology induced by the compactification $\overline{\Theta}$.
The space $\overline{\Theta}$ is called the \emph{parabolic compactification} of the irrationals because of its compatibility with sector renormalization and parabolic enrichment; refer to \cite{Lim26} for details.

\begin{theorem}[Uniform Continuity \cite{DL26}]
\label{thm:uniform-continuity}
    For $\theta \in \Theta_{\textnormal{bdd}}$, the pseudo-Siegel disks $\ldots \subset \hat{Z}^{1}_\theta \subset \hat{Z}^{0}_\theta \subset \hat{Z}^{-1}_\theta$ can be constructed such that for any fixed $m$, 
    the Riemann mapping $\RS \backslash \overline{\D} \to \RS \backslash \hat{Z}^m$ fixing $\infty$ and sending $1$ to the critical point $c_0(f_\theta)$ depends uniformly continuously on $\theta$ in the $C^0$ topology. 
    Consequently, for $\theta \in \Theta$, 
    \begin{enumerate}[label=\textnormal{(\arabic*)}]
        \item pseudo-Siegel disks $\ldots \subset \hat{Z}^{1}_\theta \subset \hat{Z}^{0}_\theta \subset \hat{Z}^{-1}_\theta$ extend by continuity to all $\theta \in \Theta$ such that each $\hat{Z}^{m}_\theta$ is $f^{q_{m+1}(\theta)}_\theta$-almost invariant;
        \item the Mother Hedgehog of $f_\theta$ is equal to 
        \[
            H_\theta := \bigcap_{m\geq-1} \hat{Z}^{m}_\theta,
        \]
        and it depends uniformly continuously on $\theta$ in the Hausdorff topology;
        \item the critical point $c_0(f_\theta)$ is an accessible point of $H_\theta$ from infinity and $H_\theta \backslash \{c_0(f_\theta)\}$ is connected;
        \item the Riemann mapping $\RS \backslash \overline{\D} \to \RS \backslash H_\theta$ fixing $\infty$ and sending $1$ to $c_0(f_\theta)$ depends uniformly continuously on $\theta$ in the compact open topology.
    \end{enumerate}
\end{theorem}

Note that (3) guarantees that the Riemann mapping in (4) is uniquely defined.

\begin{remark}
    The theorem actually says more. The Mother Hedgehog $H_\theta$ is star-like, that is, it can be presented as a bouquet of uncountably many closed quasiarcs, called internal rays, centered at the fixed point $0$. The map $f_\theta$ permutes the internal rays very much like rotation by angle $\theta$ about $0$. Moreover, for $\theta \in \overline{\Theta}\backslash \Theta$, there exist pseudo-Siegel disks almost invariant under a tower of Lavaurs maps, i.e. enrichment of parabolic maps $f_{\theta'}$, $\theta' \in \Q$. In that case, the intersection of the pseudo-Siegel disks is a Mother Hedgehog completely invariant under the Lavaurs tower.
\end{remark}

In the rest of the paper, we will only discuss bounded-type quadratic Siegel disks. However, by virtue of the uniform continuity in Theorem \ref{thm:uniform-continuity}, all the results in this paper extend to all quadratic polynomials with an irrationally indifferent fixed point. Any statement regarding bounded-type Siegel disks can be replaced with the Mother Hedgehogs, and the definition of bubbles can be replaced with \emph{hoglets}, i.e. iterated preimages of Mother Hedgehogs. 

\section{The external map}
\label{sec:external}  
In this section, we translate the pseudo-Siegel bounds of $f_\theta$ near $\overline Z_\theta$ into the external coordinate $f_\theta\leadsto \fext$ defined in~\eqref{eq:dfn:f_ext}; cf, \cite[{\S11.0.2 and Remark 11.8}]{DL22}.

Let us fix $\theta \in \Theta_{\text{bdd}}$ and we will assume $0 < \theta < \frac{1}{2}$ without loss of generality. 
Denote the map $f=f_\theta$, the Siegel disk $Z=Z_\theta$, the pseudo-Siegel disks $\hat{Z}^m = \hat{Z}^m_\theta$ for $m \geq -1$, and the critical point $c_0=c_0(f_\theta)$. 

Let us fix the normalized Riemann mapping 
\[
    \psi: \RS \backslash Z \to \RS \backslash \overline{\D}, \quad \psi(\infty)=\infty \text{ and } \psi(c_0) = 1.
\]
Denote by $B^0$ the closure of $f^{-1}(Z)\backslash Z$ and denote 
\[
\Bub^0 := \psi(B^0).
\]
The \emph{external map} of $f$ is the map
\begin{equation}
    \label{eq:dfn:f_ext}
    \fext : \RS \backslash \left( \overline{\D} \cup \Bub^0 \right) \to \RS \backslash \overline{\D}, \qquad \fext:= \psi \circ f_\theta \circ \psi^{-1}.
\end{equation}
It is a holomorphic double covering map branched only at infinity. 

\begin{lemma}
    The map $\psi$ uniquely extends to a homeomorphism from the boundary of the Siegel disk $Z$ of $f$ to the unit circle $\partial \D$. 
    The external map $\fext$ extends to a $\partial \D$-symmetric holomorphic map that is well-defined and holomorphic everywhere away from $\Bub^0$ and its reflection along $\partial \D$.
\end{lemma}

\begin{proof}
    The first claim follows from the fact that $Z$ is a Jordan disk.
    The second claim follows from Schwarz reflection.
\end{proof}

This elementary lemma gives rise to a unique ``critical'' orbit 
\[
    \left\{ \crit_j \equiv \crit_j^{\ext} := \psi(c_j) \right\}_{j \in \Z}
\]
on the unit circle $\partial \D$.

We will equip $\C \backslash \overline{\D}$ with the complete conformal hyperbolic metric. 
The unit circle $\partial \D$ will be equipped with the Euclidean metric induced by the identification $\R/\Z \to \partial \D, x \mapsto e^{2\pi i x}$. 
With respect to this metric, the length of any interval $I \subset \partial \D$ will be denoted by $|I|$. Note that the Euclidean metric is different than the combinatorial metric, which is the unique normalized $f_{\text{ext}}$-invariant metric on $\partial \D$.

Throughout the rest of this paper, objects living in external coordinates will be written in sans-serif, e.g. $\crit_j$, $\Bub^0$, etc.

\subsection{Real a priori bounds}
\label{ss:real-bounds}

For $n \geq 0$, denote by $p_n/q_n = [0;a_1,\ldots,a_n]$ the $n$\textsuperscript{th} best rational approximation of $\theta = [0;a_1,a_2,a_3,\ldots]$. 

Consider the tiling $\mathcal{D}_n$ of the circle $\partial \D$ obtained by removing the set of pre-critical points $\{\crit_{-j}\}_{j=0,1,\ldots ,q_{n+1}-1}$; it is essentially $\mathfrak{D}_n$ in external coordinates.
More explicitly,
\[
    \mathcal{D}_n = \{ [\crit_{-j}, \crit_{-q_n-j}]\}_{0\leq j \leq q_{n+1}-q_{n}-1} \cup \{ [\crit_{-q_{n+1}+q_n-j}, \crit_{-j}]\}_{0\leq j \leq q_n-1}.
\]
Every tile in $\mathcal{D}_n$ has combinatorial length either $l_n$ or $l_n + l_{n+1}$.
For every $n\geq 0$, $\mathcal{D}_{n+1}$ is a refinement of $\mathcal{D}_n$, and every tile in $\mathcal{D}_n$ decomposes into at least two tiles in $\mathcal{D}_{n+2}$.

\begin{theorem}[Real a priori bounds]
\label{thm:R-APB}
    There exists a universal constant $K>1$ such that for all $n \geq 0$,
    \begin{enumerate}[label = \textnormal{(\arabic*)}]
        \item if $I$ and $J$ are two adjacent tiles in $\mathcal{D}_n$, then $K^{-1}|J| \leq |I| \leq K|J|$;
        \item if $I \in \mathcal{D}_n$ contains $J \in \mathcal{D}_{n+1}$ and $I\backslash J$ is connected, then $|I| \leq K|J|$.
    \end{enumerate}
    In particular, the maximal diameter of tiles in $\mathcal{D}_n$ tends to zero uniformly exponentially fast as $n \to \infty$.
\end{theorem}

\begin{proof}
    Consider two neighboring intervals $I$ and $J$ in $\mathcal{D}_n$. If both are combinatorial intervals of the same level, then (1) follows from Theorem \ref{thm:uniform-external}. Else, one of them, say $I$, has combinatorial length $l_n+l_{n+1}$ and $J$ has length $l_n$. Split $I$ into intervals $I_1$ and $I_2$ of combinatorial lengths $l_{n+1}$ and $l_n$ respectively where $I_2$ is touching $J$. Let $L$ be the unique level $n$ interval containing $I_1$ that is disjoint from $I_2$. Again, by Theorem \ref{thm:uniform-external}, we have $|I_1| \leq |L| \asymp |I_2| \asymp |J|$, which implies (1).

    Next, let us prove item (2). Pick $I \in \mathcal{D}_n$. 
    Let $J_1,\ldots, J_{k}$ be the sequence of tiles in $\mathcal{D}_{n+1}$ that are contained in $I$, labelled in consecutive order. 
    The total number $k$ is either $a_{n+2}$ or $a_{n+2}+1$.
    If $k \leq 2$, it is immediate from item (1) that for $1 \leq i \leq k$, $|J_i|$ is uniformly comparable to $|J|$.
    In the general case, \cite[Lemma 10.5]{DL22} (as part of the proof of Theorem \ref{thm:uniform-external}) states that the width of curves in $\C \backslash \overline{\D}$ that start from $J_1$ and end at $J_k$ is uniformly bounded from below. Therefore, there is a universal constant $K'>0$ such that
    \[
    \min\{|J_1|,|J_k|\} \geq K' \sum_{2\leq i \leq k-1} |J_i|.
    \]
    Together with item (1) applied to $J_1,J_2$ and $J_{k-1}, J_k$, we have the estimate $|I| \leq K \min\{|J_1|,|J_k|\}$ for some universal constant $K>1$.
\end{proof}

This result hints at the similarity between $\fext$ and critical circle maps. 
For critical circle maps, real a priori bounds were established by Herman and \'Swi\k{a}tek \cite{H87,Sw87} and such bounds are the foundation of the renormalization theory of critical circle maps as well as bounded-type quadratic Siegel disks.

\subsection{Regularity of $\psi$}

For $m \geq -1$, denote 
\[
    \hat{\mathsf{Y}}^m := \psi(\partial \hat{Z}^{m})
    \quad \text{ and } \quad
    \hat{\mathsf{Y}} := \hat{\mathsf{Y}}^{-1}.
\]

\begin{lemma}
\label{lem:ext-pseudo-siegel}
    There exists a universal constant $K>1$ such that for all $m \geq -1$, $\hat{\mathsf{Y}}^m$ is a $K$-quasicircle.
\end{lemma}

\begin{proof}
    Below, we will only prove the lemma for EGM rotation number $\theta$. The general case follows from continuity.
    
    Let us recall the construction of $\hat{\mathsf{Y}}^m$. For all sufficiently high $m$, $\hat{\mathsf{Y}}^m$ coincides with the unit circle $\partial \D$. In general, there are two cases for every $m \geq -1$.
    If $a_{m+2}< \Mbold$, then $\hat{\mathsf{Y}}^m$ is set to be equal to $\hat{\mathsf{Y}}^{m+1}$. If $a_{m+2} \geq \Mbold$, then we set $\hat{\mathsf{Y}}^m$ to be the external boundary of $\hat{\mathsf{Y}}^{m+1} \cup \bigcup_{I \in \mathcal{D}_m} \mathsf{b}_I$, where each $\mathsf{b}_I$ is a hyperbolic geodesic of $\C \backslash \overline{\D}$ whose endpoints are contained in $I$ and partition $I$ into three subintervals of comparable Euclidean length.
    
    Pick any two distinct points $x$ and $y$ on $\hat{\mathsf{Y}}^m$ such that $|\text{arg}(x)-\text{arg}(y)|< \frac{\pi}{2}$. 
    By Theorem \ref{thm:R-APB} (1) and some elementary trigonometry, the shortest arc $[x,y]$ in $\partial \hat{\mathsf{Y}}^m$ with endpoints $x$ and $y$ satisfies
    \[
        \text{diam}[x,y] \leq C |x-y|
    \]
    for some universal constant $C>1$. 
    This implies that $\hat{\mathsf{Y}}^m$ is indeed a uniform quasicircle.
\end{proof}

\begin{corollary}
\label{cor:new-psi}
    There exists a universal constant $K>1$ such that the univalent map $\psi : \RS \backslash \hat{Z} \to \C$ extends to a $K$-quasiconformal homeomorphism $\psi_{\textnormal{qc}} : \RS \to \RS$.
\end{corollary}

\begin{proof}
    By Theorem \ref{thm:pseudo-siegel-bounds} and Lemma \ref{lem:ext-pseudo-siegel}, both $\partial \hat{Z}$ and $\hat{\mathsf{Y}}$ are uniform quasicircles. 
    The claim then follows from standard quasiconformal interpolation.
\end{proof}

\subsection{Bubbles and pseudo-bubbles}
\label{ss:bubbles-and-pseudo-bubbles}

From now on, let us work in the dynamical plane of $\fext$. 
The \emph{bubble} of $\fext$ of generation $0$ is $\Bub^0$. 
It is contained in a \emph{pseudo-bubble} $\hat{\Bub}^0 = \psi(\hat{B}^0)$, which consist of $\Bub^0$ together with a collection of dams filled in. 
Each dam on the boundary of $\Bub^0$ is protected by definite collars (Section  \ref{ss:pseudo-siegel}).
 
The image of the boundary of $\hat{\Bub}^0$ under $\fext$ is $\hat{\mathsf{Y}}$.

A \emph{bubble} $\Bub$ of $\fext$ of generation $g>0$ is a connected component of $\fext^{-g}(\Bub^0)$, and it is contained in a \emph{pseudo-bubble} $\hat{\Bub}$, which is the unique connected component of $\fext^{-g}\left(\hat{\Bub}^0\right)$ containing $B$. 
The \emph{root} of a (pseudo-)bubble of generation $g\geq 0$ is the unique point $\crit$ in it such that $\fext^{g}(\crit)=\crit_0$. 

Let us equip $\C \backslash \overline{\D}$ with the hyperbolic metric. 
For any $k \in \N$, let us denote by $\Gamma_k$ the infinite radial ray in $\C \backslash \overline{\D}$ landing at $\crit_{-k}$. 
Given $K>0$, we consider the ``hyperbolic cone''
\[
    \mathsf{C}_k(K) := \{ x \in \C \backslash \overline{\D} \: : \: \dist_{\text{hyp}}(x, \Gamma_k) \leq K\} \cup \{\crit_{-k}\}.
\]

\begin{lemma}
\label{lem:main-bubble}
    There exists a universal $K\geq 1$ such that the pseudo-bubble $\hat{\Bub}^0$ is a $K$-quasicircle with diameter
    \[
        \frac{1}{K} \leq \textnormal{diam}(\hat{\Bub}^0) \leq K
    \]
    and $\hat{\Bub}^0$ is contained in the cone $\mathsf{C}_0(K)$.
\end{lemma}

\begin{proof}
    The preimage of $\hat{Z}$ is the union of the pseudo-bubble $\hat{B}^0$ and a pseudo-Siegel disk $\tilde{Z}$. 
    The symmetric difference between $\tilde{Z}$ and $\hat{Z}$ is controlled by the collars of the dams in the boundary of $\hat{Z}$. 
    Since the collars avoid the critical value, they must be disjoint from the collars of $\hat{B}^0$. 
    Hence, $\hat{B}^0$ intersects $\hat{Z}$ only at its root, the critical point of $f$.
    
    Consider the quasiconformal extension $\psi_{\textnormal{qc}}$ described in Corollary \ref{cor:new-psi}. 
    Since the pseudo-bubble $\hat{\Bub}^0$ of $\fext$ is the image of the pseudo-bubble $\hat{B}^0$ of $f$ under $\psi_{\textnormal{qc}}$, then $\hat{\Bub}^0$ is a uniform quasidisk. The diameter estimate then follows from the fact that $\hat{B}^0$ and $\tilde{Z}$ have the same diameter. It remains to show that $\hat{\Bub}^0$ is contained in some definite cone with vertex $\crit_0$.

    There exists a quasiconformal homeomorphism of $\phi : \C\to \C$ that maps $\hat{Z}$ onto the closed disk $\{|z+i|\leq 1\}$, sends $0$ to $-i$, and the critical value $c_1$ to $0$. 
    Consider the quasiconformal half plane $V = \{ \text{Im}\phi(z)>0\}$ whose closure intersects $\hat{Z}$ only at the critical value. 
    Then, $f^{-1}(V)$ consists of two infinite sectors $S_+ \cup S_-$ with a common vertex at the critical point $\crit_0$. 
    Consider a quasiconformal homeomorphism $\Phi : \C \to \C$ such that $ \Phi(z)^2 = \phi\circ f(z)$. 
    The map $\Phi$ sends $f^{-1}(V)$ to $\{x+iy \: | \: xy>0\}$, a union of radially opposite right-angled infinite straight sectors. 
    
    Since the closure of $f^{-1}(V)$ separates $Z$ from the pseudo-bubble $\hat{B}^0$, then we must have $\text{dist}(z,Z) \asymp |z-c_0|$ for every point $z$ in $\hat{B}^0$. 
    This implies that every point in $\hat{B}^0$ must be of some uniformly bounded hyperbolic distance away from the infinite hyperbolic geodesic $\psi^{-1}(\Gamma_0)$ of $\C \backslash \overline{Z}$ landing at the critical point. 
    Hence, the pseudo-bubble $\hat{\Bub}^0$ must be contained in the cone $\mathsf{C}_0(K')$ for some universal constant $K'>0$.
\end{proof}

For $k \geq 0$, let us denote by $\Bub^k$ (resp. $\hat{\Bub}^k$) the unique primary bubble (resp. pseudo-bubble) rooted at $\crit_{-k}$. 
The next lemma quantifies how $\hat{\Bub}^k$ meets $\partial \D$ at a definite angle.

\begin{lemma}
\label{lem:control-near-roots}
    There exists a universal constant $K\geq 1$ such that the following holds for every integer $k \geq 0$. 
    Let $n \geq 0$ be such that $q_n \leq k < q_{n+1}$ and let $\sigma_k$ be the unique combinatorially shortest hyperbolic geodesic of $\C \backslash \overline{\D}$ that connects $\crit_{-k-q_{n+1}}$ and $\crit_{-k+q_{n+1}}$. 
    Then, the intersection between the pseudo-bubble $\hat{\Bub}^k$ and the domain in $\RS \backslash \overline{\D}$ bounded by $\sigma_k$ is contained in the cone $\mathsf{C}_k(K)$. 
\end{lemma}

\begin{proof}
    Let $\crit_{-l_k}$ and $\crit_{-r_k}$ be the two critical points of $\fext^k$ on the left and right of $\crit_{-k}$ respectively. 
    Since $[\crit_{-l_k},\crit_{-r_k}]$ compactly contains the interval $[\crit_{-k-q_{n+1}}, \crit_{-k+q_{n+1}}]$, we can apply Koebe distortion theorem together with Lemma \ref{lem:main-bubble} to prove the claim.
\end{proof}


\section{Curve families from bubbles}
\label{sec:curve-families}

Let us work with the dynamical plane of $\fext$. 
Given any integer $k \geq 0$, let us denote
\[
    \Uext_k := \C \backslash \left( \overline{\D} \cup \Bub^k \right),
\]
which is a topological disk punctured at infinity. 
There exists a unique $n \in \N$ satisfying $q_{n} \leq k < q_{n+1}$.
Let us define the curve family $\mathcal{F}_k$ as the set of proper arcs in $\Uext_k$ that start from a point on $\Bub^k$ and \emph{protect} either $[\crit_{-k-q_n},\crit_{-k}] \subset \partial \D$ or $[\crit_{-k},\crit_{-k+q_n}] \subset \partial \D$.
See Figure \ref{fig:curve-family}.

Here, the term \emph{protect} means the following.
A proper arc $\gamma$ on an open punctured disk $X \subset \RS$ is said to protect a subset $K$ of $\overline{X}$ if $K$ is contained in the closure of the unique disk enclosed by $\gamma \cup \partial X$ away from the puncture.

In this section, we will prove the following theorem in the near-degenerate regime.

\begin{figure}
    \centering
\begin{tikzpicture}
    \node[anchor=south west,inner sep=0] (image) at (0,0) {\includegraphics[width=0.95\linewidth]{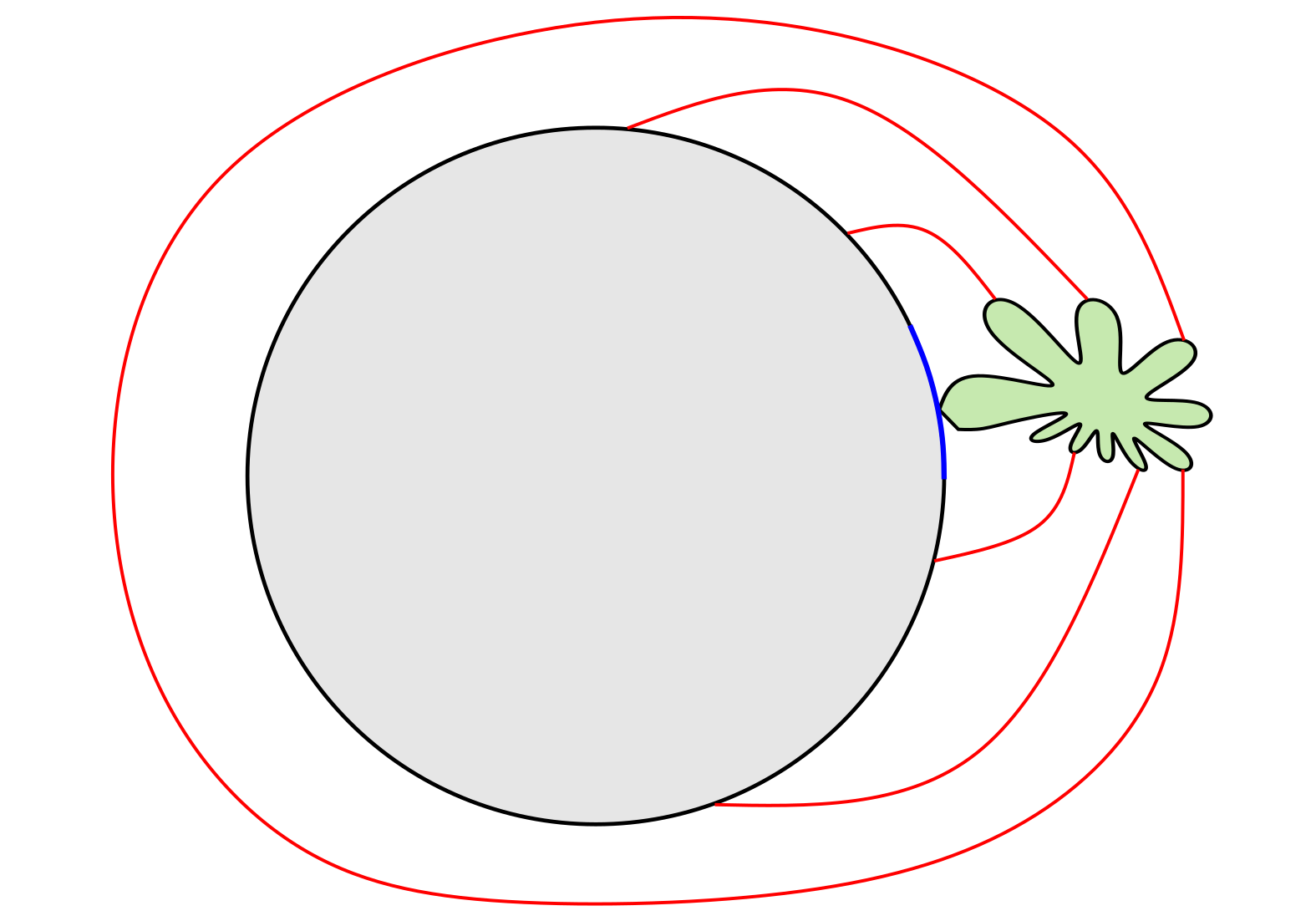}};
    \begin{scope}[
        x={(image.south east)},
        y={(image.north west)}
    ]
    
    \node [black] at (0.47,0.48) {\Large{$\mathbb{D}$}};
    \node [blue] at (0.665,0.475) {$\crit_{-k-q_{n}}$};
    \node [blue] at (0.7175,0.475) {$\bullet$};
    \node [blue] at (0.675,0.56) {$\crit_{-k}$};
    \node [blue] at (0.712,0.555) {$\bullet$};
    \node [blue] at (0.635,0.66) {$\crit_{-k+q_{n}}$};
    \node [blue] at (0.69,0.65) {$\bullet$};
    \node [black] at (0.835,0.57) {$\Bub^k$};
    \node [red] at (0.835,0.25) {$\mathcal{F}_k$};
    \end{scope}
\end{tikzpicture}
    \caption{The curve family $\mathcal{F}_k$}
    \label{fig:curve-family}
\end{figure}
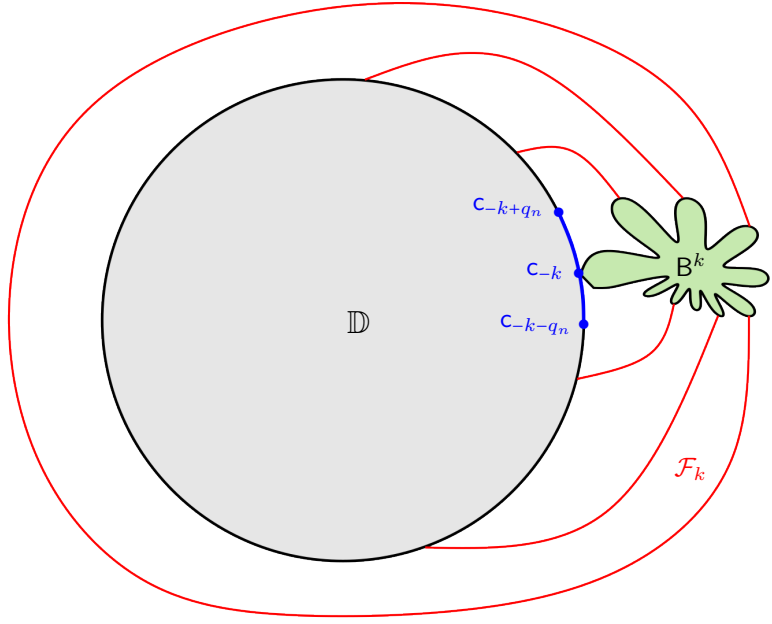

\begin{theorem}
\label{thm:curve-families}
    There exists a universal $K >1$ such that for all $k \geq 0$, we have 
    \[
    W(\mathcal{F}_k) \leq K.
    \]
\end{theorem}

Let us outline the main steps of the proof; see also~\S\ref{ss:summary}. We argue by contradiction, assuming that $\mathcal F_k$ is sufficiently wide for some $k$.

\begin{enumerate}[label=\textnormal{(\arabic*)}]
    \item\label{item:outline.1}  By removing $O(t)$ curves from $\mathcal F_k$, we can assume that the remaining curves in  $\mathcal F_k$ land $t$ combinatorial intervals away from $\Bub^k$, see Lemma~\ref{lem:local-width}. This is an application of Theorem~\ref{thm:R-APB}, real a priori bounds.
    \item Proposition~\ref{prop:weak-improvement} (almost monotonicity) states that most curves in $\mathcal F_k=\mathcal F_{a+b}$ can be pushed forward toward $\mathcal F_{a}$ whenever $a+b$ and $a$ represent ``neighboring'' renormalization scales. See also the notion of the $\alpha'$-Dynamics in~\S\ref{sss:alpha'.Dynam}. 
    \item\label{item:outline.3} Lemma~\ref{lem:blocking} describes how families $\mathcal F_t$ with $q_n \le t\le 2q_{n}-1$ block one another. The lemma asserts the existence of $\mathcal{F}_{\mathbf{t}}$ so that most curves in $\mathcal{F}_{\mathbf{t}}$ emerging from $\Bub^{\mathbf{t}}$ submerge in neighboring bubbles $\Bub^{l({\mathbf{t}})}\cup \Bub^{r({\mathbf{t}})}$ with respect to the cyclic order. Here, Item~\ref{item:outline.1} allows us to assume that most curves in $\mathcal{F}_{\mathbf{t}}$ land sufficiently far from $\Bub^{\mathbf{t}}$.
    
    \item By Proposition~\ref{prop:weak-improvement}, we assume that $\mathcal{F}_{\mathbf{t}}$ in Item~\ref{item:outline.3} is almost as wide as the original $\mathcal F_k$.
     \item\label{item:outline.5} The final step is carried out in the proof of Theorem~\ref{thm:amplification}. The submergence of $\mathcal{F}_{\mathbf{t}}$ into $\Bub^{l({\mathbf{t}})}\cup \Bub^{r({\mathbf{t}})}$ produces a strict amplification that can then be pushed forward to a shallower scale:
    \[
    W(\mathcal{F}_{q_{n-2}}) \geq \frac{4}{3} W(\mathcal{F}_{\mathbf{t}}) - O(1).
    \]
     We distinguish into three cases according to whether $\Bub^{\mathbf{t}}$ has a smaller generation than $\Bub^{l({\mathbf{t}})}$ or $\Bub^{r({\mathbf{t}})}$.
\end{enumerate}

Before going through the details, the reader may wish to review Appendix \ref{sec:appendix} for some background on extremal width, rectangles, and laminations.

\subsection{Almost monotonicity}
\label{ss:preparation}
First, observe that the analysis of $\mathcal{F}_k$ can be reduced to disjoint simple proper curves.

\begin{lemma}
\label{lem:canonical-lamination-Dk}
    For any $k \geq 0$, there exists a lamination $\mathcal{L}_k \subset \mathcal{F}_k$ such that
    \[
        W(\mathcal{L}_k) \leq W(\mathcal{F}_k) \leq W(\mathcal{L}_k) + O(1).
    \]
\end{lemma}

\begin{proof}
    This is a straightforward application of Lemma \ref{lem:canonical-lamination} applied to the punctured disk $\Uext_k$.
\end{proof}

For any pair of integers $k \geq 0$ and $m \geq 1$, denote by $\Uext_k^m$ the unbounded connected component of $\fext^{-m}(\Uext_k)$. 
Each $\Uext_k^m$ is a disk punctured at infinity. 
Later on, we will be repeatedly applying Lemma \ref{lem:univalent-pushforward} to the degree $2^m$ covering map
\[
    \fext^m : \Uext_k^m \to \Uext_k.
\]
Note that $\fext : \Uext_0 \to \C \backslash \overline{\D}$ is also a double covering map.

The following lemma quantifies the geometric separation between bubbles and critical points on a deeper scale.

\begin{lemma}[Control near critical points]
\label{lem:control-critical-pt}
    There exist a universal even integer $\mathbf{m} \geq 2$ and a universal constant $C >1$ such that the following holds
    for any pair of integers $j, n \geq 0$ with $j \leq q_{n+1}-2$. 
    Let $\Rext_j$ be the proper conformal rectangle in $\C \backslash \overline{\D}$ that has horizontal sides $[\crit_{-j-q_{n-2\mathbf{m}}},\crit_{-j-q_{n-\mathbf{m}}}] \subset \partial \D$ and $[\crit_{-j-q_{n+\mathbf{m}}},\crit_{-j-q_{n+2\mathbf{m}}}] \subset \partial \D$, has vertical sides being the hyperbolic geodesics of $\C \backslash \overline{\D}$, and protects the critical point $\crit_{-j}$.
    Then,
    \begin{enumerate}[label=\textnormal{(\arabic*)}]
        \item $\Rext_j$ has width 
        \[
            \frac{1}{C} \leq W(\Rext_k) \leq C;
        \]
        \item for every $k \in \N$ with $j+1 \leq k \leq q_{n+1}-1$, the bubble $\Bub^k$ does not intersect the inner vertical side of $\Rext_j$.
    \end{enumerate}
\end{lemma}

\begin{proof}
    The estimate on the width of $\Rext_j$ follows from Theorem \ref{thm:R-APB}. 
    Suppose property (2) fails, so there exists some integer $k \in  \{ j+1,\ldots, q_{n+1}-1\}$ such that the bubble $\Bub^k$ intersects the inner vertical boundary $\delta_j$ of $\Rext_j$. 
    We will show that this is impossible when $\mathbf{m}$ is sufficiently high.

    Firstly, by Lemma \ref{lem:main-bubble} and Theorem \ref{thm:R-APB}, there exists a universal even integer $\nu \geq 2$ such that the primary bubble $\Bub^0$ is disjoint from the hyperbolic geodesic $\gamma_i$ joining $\crit_{i-q_{n+\nu}}$ and $\crit_{i+q_{n+\nu}}$ for all $i \in \{1,\ldots,q_{n+1}\}$. 
    The disk $S_{k-j} \subset \C \backslash \overline{\D}$ enclosed by $\gamma_{k-j}$ is well separated from all the critical values of $\fext^{k}$ with the exception of $\crit_{k-j}$ itself. 
    Hence, by Koebe distortion theorem, the hyperbolic geodesic $\delta_j$ must be sent into $S_{k-j}$ by $\fext^{k}$ if we choose $\mathbf{m} \gg \nu$.
    Since $\Bub^k$ is assumed to intersect $\delta_j$, then $\Bub^0$ has to intersect $\gamma_{k-j}$, which is a contradiction.
\end{proof}

\begin{lemma}
\label{lem:local-width}
    Consider integers $t \geq 1$, $n \geq 1$, and $k \in \{ q_{n},\ldots,q_{n+1} -1 \}$. Let 
    \[
    I_{-t}, \quad I_{-t+1},\quad \ldots \quad I_{-1}, \quad I_0, \quad I_1, \quad \ldots \quad I_{t-1}, \quad I_t
    \]
    denote the unique sequence of consecutive connected components of $\partial \D \backslash \{ \crit_{-j}\}_{0 \leq j \leq k-1}$ such that $I_0$ contains $\crit_{-k}$ and for $i \in \{-t+1,\ldots,t-1\}$, $I_{i-1}$ is on the left of $I_i$ and $I_{i+1}$ is on the right of $I_i$. 
    The family of proper curves in $\Uext_k$ that start from the bubble $\Bub^k$ and land on $\bigcup_{i=-t}^t I_i \backslash [\crit_{-k-q_{n+2}}, \crit_{-k+q_{n+2}}]$ has width $O(t)$. 
\end{lemma}

\begin{proof}
    For simplicity, let us prove the lemma for $t=1$; the argument below can be easily iterated to cover the general case $t>1$. 
    
    There exists integers $a,b \in \{0,1,\ldots,k-1\}$ such that $I_1=(\crit_{-a},\crit_{-b})$ and $\crit_{-a}$ is an endpoint of $I_0$. 
    Below, we will show that the family $\mathcal{G}$ of proper curves in $\Uext_k$ that start from the bubble $\Bub^k$ and land at a point between $\crit_{-k}$ and $\crit_{-b}$ has bounded width.
    (The treatment for those landing on the left of $\crit_{-k}$ is analogous.)
    By Lemma \ref{lem:canonical-lamination}, there exists a lamination $\mathcal{G}'$ in $\Uext_k$ connecting $\Bub^k$ and $[\crit_{-k},\crit_{-b}]$ with width $\geq W(\mathcal{G}) - O(1)$.

    Consider the constant $\mathbf{m}$ and the rectangles $\Rext_a$ and $\Rext_b$ near $\crit_{-a}$ and $\crit_{-b}$ respectively from Lemma \ref{lem:control-critical-pt}. 
    For any $j \in \{a,b,k\}$, denote by $J_j$ the interval of points on $\partial \D$ that are of combinatorial distance at most $l_{n+\mathbf{m}}$ from $\crit_{-j}$.
    Denote
    \[
        I'_0 := (\crit_{-k}, \crit_{-a}) \backslash (J_k \cup J_a), \qquad I'_1 := I_1 \backslash (J_a \cup J_b).
    \]
    For $i \in \{0,1\}$, denote by $\mathcal{S}_i$ the set of leaves of $\mathcal{G}'$ that land on $I'_i$. 
    Every leaf of the lamination $\mathcal{G}' \backslash (\mathcal{S}_0 \cup \mathcal{S}_1)$ lands on $J_a \cup J_b$, so it must cross either $\Rext_a$ or $\Rext_b$. 
    Therefore,
    \[
        W(\mathcal{S}_0) + W(\mathcal{S}_1) \geq W(\mathcal{G}') - O(1).
    \]

    For $i \in \{0,1\}$, consider the restriction $\mathcal{S}^{\res}_i$ of $\mathcal{S}_i$ into the punctured disk $\Uext_0^k$ such that each leaf of $\mathcal{S}^{\res}_i$ starts from a point on $I'_i$ and ends somewhere on the boundary of $\Uext_0^k$. 
    By Lemma \ref{lem:univalent-pushforward}, there is a sub-lamination of $\mathcal{S}^{\res}_i$ that can be univalently pushed forward by $\fext^{k+1}$ onto a proper lamination $\mathcal{S}^{\new}_i$ in $\C \backslash \overline{\D}$ with width 
    \[
    W(\mathcal{S}^{\new}_i) \geq W(\mathcal{S}^{\res}_i) - O(1).
    \]
    Every leaf of $\mathcal{S}^{\new}_i$ starts from the interval $\fext^{k+1}(I'_i)$, which has combinatorial length at most $l_{n-1}+l_n$, and protects a neighboring interval of combinatorial length $l_{n+\mathbf{m}}$. 
    By Theorem \ref{thm:R-APB} and Proposition \ref{prop:log-rule}, we have $W(\mathcal{S}^{\new}_i) = O(1)$ for each $i$. 
    Therefore, $W(\mathcal{G}') = O(1)$.
\end{proof}

\begin{proposition}[Almost Monotonicity]
\label{prop:weak-improvement}
    There exists some universal constant $C>0$ such that for any integers $a$, $b$, and $n\geq 0$ satisfying $q_{n} \leq a < a+b \leq q_{n+1}$, 
    \begin{equation}
    \label{eqn:weak-improvement}
       W(\mathcal{F}_a) \geq   W(\mathcal{F}_{a+b})- C.
    \end{equation}
\end{proposition}

\begin{proof}
    Consider the lamination $\mathcal{L}_{a+b}$ from Lemma \ref{lem:canonical-lamination-Dk}. 
    Let $\mathcal{L}^b_{a+b}$ be the restriction of $\mathcal{L}_{a+b}$ in $\Uext^b_a$ which consists of curves with starting point on the bubble $\Bub^{a+b}$. 
    By Lemma \ref{lem:univalent-pushforward}, there exists a sub-lamination of $\mathcal{L}^{b}_{a+b}$ that is univalently pushed forward to a proper lamination $\mathcal{L}'_a$ of $\Uext_a$, and it has width
    \[
        W(\mathcal{F}_{a+b}) \leq W(\mathcal{L}'_a) + O(1).
    \]

    If $a+b < q_{n+1}$, then $\mathcal{L}'_a$ is a subset of $\mathcal{F}_a$ and so the desired inequality (\ref{eqn:weak-improvement}) holds. 
    Let us now assume that $a+b = q_{n+1}$. 
    We can split $\mathcal{L}'_a$ into two disjoint sub-laminations $\mathcal{L}''_a$ and $\mathcal{L}'''_a$, where $\mathcal{L}''_a$ is contained in $\mathcal{F}_a$ and every leaf of $\mathcal{L}'''_a$ lands on $[\crit_{-a-q_{n}},\crit_{-a+q_{n+1}}] \cup [\crit_{-a-q_{n+1}},\crit_{-a+q_{n}}]$. However, $\mathcal{L}'''_a$ has width $O(1)$ by virtue of Lemma \ref{lem:local-width}, so then we again arrive at the inequality (\ref{eqn:weak-improvement}).
\end{proof}

We know from Theorem \ref{thm:R-APB} and Lemma \ref{lem:main-bubble} that 
\[
    W(\mathcal{F}_0) = O(1).
\]
The corollary above implies that the width of $\mathcal{F}_{q_n}$ can only grow at most linearly:
\begin{align}
    \label{eqn:W-linear}
    W(\mathcal{F}_{q_n}) = O(n).
\end{align}
In the next subsection, we will remove the dependence on $n$.

\subsection{Amplification}

Let us fix $n\geq 2$. 
Consider the set of points 
\begin{equation}
\label{lvl-n-list}
    \left\{ \crit_{-s} \: | \: s=q_n, q_n+1,\ldots, 2q_{n}-1 \right\}.
\end{equation}
It defines a tiling of $\partial \D$ where every tile has combinatorial length equal to either $l_{n-1}$ or $l_{n-1}+l_{n}$.

For every integer $t$ in $\{q_{n}, \ldots, 2q_{n}-1\}$, denote by $l(t)$ (resp. $r(t)$) be the integer in $\{q_{n}, \ldots, 2q_{n}-1\}$ such that
$\crit_{-l(t)}$ (resp. $\crit_{-r(t)}$) is the point directly on the left (resp. right) of $\crit_{-t}$ among the critical points on the list (\ref{lvl-n-list}).
One can make this more explicit.
After swapping $l(t)$ and $r(t)$ if necessary, we have
\[
l(t) := \begin{cases}
    t+ (q_n-q_{n-1}), & \text{ if } t < q_{n} + q_{n-1}, \\
    t-q_{n-1}, & \text{ if otherwise}.
\end{cases}
\]
and
\[
r(t) := \begin{cases}
    t+q_{n-1}, & \text{ if } t < 2q_{n}-q_{n-1}, \\
    t-(q_n-q_{n-1}), & \text{ if otherwise}.
\end{cases}
\]
We will also denote 
\[
l^2(t):= l(l(t)) \qquad \text{and} \qquad r^2(t):=r(r(t)).
\]

\begin{lemma}[Blocking]
\label{lem:blocking}
    For every $n\geq 2$, there exists an integer ${\mathbf{t}} = \mathbf{t}(n)$ between $q_n$ and $2q_{n}-1$ and a lamination $\mathcal{G}_{\mathbf{t}} \subset \mathcal{F}_{\mathbf{t}}$ with the following properties.
    \begin{enumerate}[label=\textnormal{(\arabic*)}]
        \item $W(\mathcal{G}_{\mathbf{t}}) \leq W(\mathcal{F}_{\mathbf{t}}) \leq W(\mathcal{G}_{\mathbf{t}}) + O(1)$.
        \item Every leaf of $\mathcal{G}_t$ lands at a point outside the interval $[\crit_{-l^2({\mathbf{t}})}, \crit_{-r^2({\mathbf{t}})}]$.
        \item Every leaf of $\mathcal{G}_{\mathbf{t}}$ intersects either $\Bub^{l({\mathbf{t}})}$ or $\Bub^{r({\mathbf{t}})}$ or both.
    \end{enumerate}
\end{lemma}

The argument below is purely topological.

\begin{proof}
    Pick $t \in \{q_n, \ldots,2q_n-1\}$. Each of the intervals $(\crit_{-l(t)}, \crit_{-t})$ and $(\crit_{-t}, \crit_{-r(t)})$ contains at most one critical point of $\fext^{t}$. 
    Consider the lamination $\mathcal{L}_t$ described in Lemma \ref{lem:canonical-lamination-Dk}.
    By Lemma \ref{lem:local-width}, the set $\mathcal{E}_t$ of leaves of $\mathcal{L}_t$ satisfying (2) must satisfy (1). 
    Denote by $\mathcal{E}_{t,l}$ (resp. $\mathcal{E}_{t,r}$) the set of leaves of $\mathcal{E}_{t}$ that avoid the bubbles $\Bub^{l(t)} \cup \Bub^{r(t)}$ and protect the interval between $\crit_{-t}$ and $\crit_{-l^2(t)}$ (resp. $\crit_{-r^2(t)}$). 
    We will show that $W(\mathcal{E}_{t,l} \cup \mathcal{E}_{t,r}) \leq 4$ for some $t$, so that the new lamination 
    \[
        \mathcal{G}_t := \mathcal{E}_t \backslash (\mathcal{E}_{t,l} \cup \mathcal{E}_{t,r})
    \]
    satisfies (3). To proceed, suppose for a contradiction that for all $t \in \{q_n, \ldots,2q_n-1\}$, $W(\mathcal{E}_{t,l} \cup \mathcal{E}_{t,r}) > 4$. 
    
    \begin{figure}
    \centering
\begin{tikzpicture}
    \node[anchor=south west,inner sep=0] (image) at (0,0) {\includegraphics[width=1\linewidth]{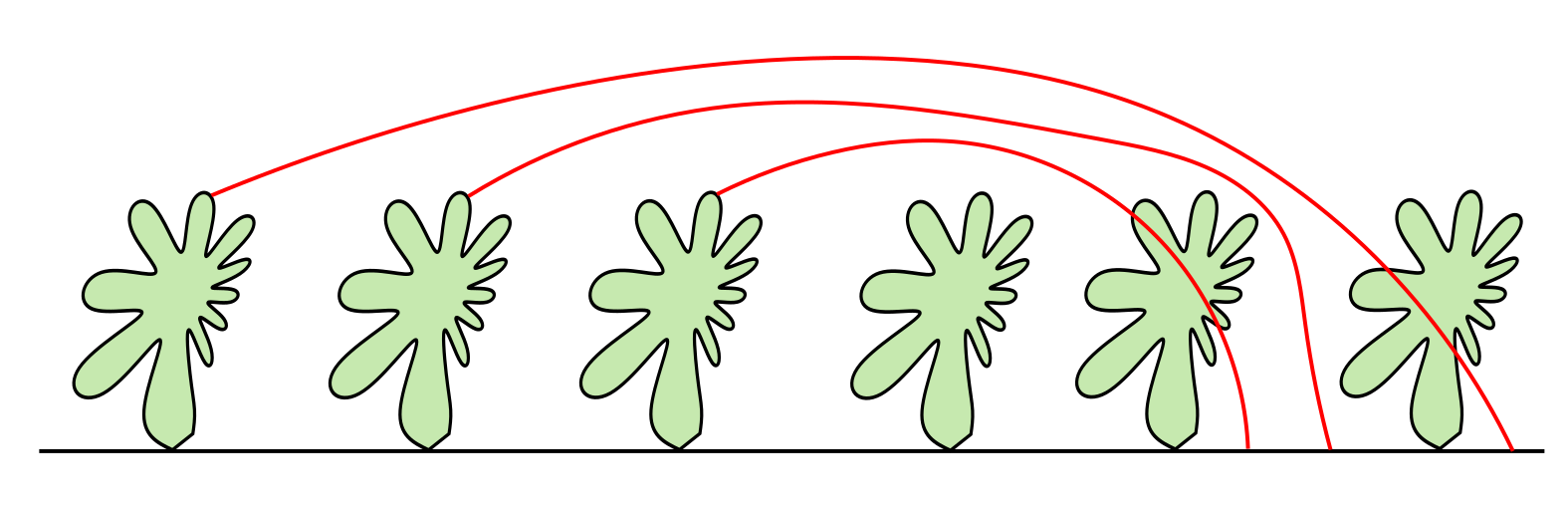}};
    \begin{scope}[
        x={(image.south east)},
        y={(image.north west)}
    ]
    
    \node [blue] at (0.11,0.05) {$\crit_{-t_0}$};
    \node [blue] at (0.11,0.11) {\small $\bullet$};
    \node [blue] at (0.273,0.05) {$\crit_{-t_1}$};
    \node [blue] at (0.273,0.11) {\small $\bullet$};
    \node [blue] at (0.433,0.05) {$\crit_{-t_2}$};
    \node [blue] at (0.433,0.11) {\small $\bullet$};
    \node [blue] at (0.605,0.05) {$\crit_{-t_3}$};
    \node [blue] at (0.605,0.11) {\small $\bullet$};
    \node [blue] at (0.749,0.05) {$\crit_{-t_4}$};
    \node [blue] at (0.749,0.11) {\small $\bullet$};
    \node [blue] at (0.918,0.05) {$\crit_{-t_5}$};
    \node [blue] at (0.918,0.11) {\small $\bullet$};
    \node [black] at (0.03,0.05) {$\partial\D$};
    \node [red] at (0.2,0.76) {$\mathcal{E}_{t_0,r}$};
    \node [red] at (0.37,0.63) {$\mathcal{E}_{t_1,r}$};
    \node [red] at (0.53,0.61) {$\mathcal{E}_{t_2,r}$};
    \end{scope}
\end{tikzpicture}
    \caption{Illustration of the proof of Lemma \ref{lem:blocking}. 
    Most curves in $\mathcal{F}_{t_3}$ are forced to cross either $\mathcal{E}_{t_2,r}$, $\Bub^{t_2}$, or $\Bub^{t_4}$.}
    \label{fig:blocking}
\end{figure}
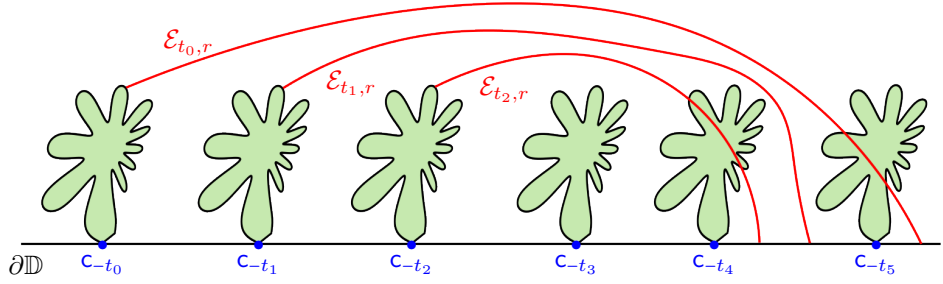

    We start with any $t_0 \in \{q_n,\ldots, 2q_n-1\}$ and assume $W(\mathcal{E}_{t_0,r}) >2 $.
    There exists a unique leaf $\gamma_0$ of $\mathcal{E}_{t_0,r}$ that defines the outer $1$-buffer. 
    (See Appendix \ref{sec:appendix} for details.)
    More precisely, if we denote by $A_0$ the bounded connected component of $\Uext_k \backslash \gamma_0$, then $\mathcal{E}_{t_0,r}$ splits $\gamma_0$ into two sub-laminations $\mathcal{E}_{t_0,r}^+$ and $\mathcal{E}_{t_0,r}^-$ where $\mathcal{E}_{t_0,r}^-$ consists of leaves that are contained in $A_0$ and 
    \[
        W(\mathcal{E}_{t_0,r}^+) = 1, \qquad W(\mathcal{E}_{t_0,r}^-) = W(\mathcal{E}_{t_0,r}) -1 > 1.
    \]
    The disk $A_0$ contains a maximal non-empty string of consecutive bubbles $\Bub^{t_1}$, $\Bub^{t_2}$, $\Bub^{t_{3}}$, $\ldots \Bub^{t_{k_0}}$ where $t_{i}=r(t_{i-1})$ for all $i =1,\ldots, k_0$ for some $k_0 \geq 1$.
    
    By our assumption, $\mathcal{E}_{t_1,l}$ is disjoint from $\Bub^t$ so $\mathcal{E}_{t_1,l}$ must cross $\mathcal{E}_{t_0,r}^-$, which implies that $W(\mathcal{E}_{t_1,l}) \leq 1$. Therefore, $W(\mathcal{E}_{t_1,r}) > 3$. Similar to $\gamma_0$, we consider the unique leaf $\gamma_1$ that splits $\mathcal{E}_{t_1,r}$ into the outer $1$-buffer $\mathcal{E}_{t_1,r}^+$ and the complement $\mathcal{E}_{t_1,r}^-$. Since both $\mathcal{E}_{t_1,r}^-$ and $\mathcal{E}_{t_0,r}^+$ have width at least one, they cannot cross and so the endpoints of $\gamma_1$ must be contained in $A_0$. Let us define $A_1$ to be the bounded connected component of $\Uext_k \backslash \gamma_1$. Then, $\partial A_1 \cap \partial \D $ is contained in $\partial A_0 \cap \partial \D \backslash [\crit_{-t_0},\crit_{-t_1})$. Therefore, the disk $A_1$ contains a maximal non-empty string of consecutive bubbles $\Bub^{t_2}$, $\Bub^{t_3}$, $\ldots$, $\Bub^{t_{k_1}}$ for some integer $k_1 \in [2, k_0]$.

    Repeat the argument from the previous paragraph inductively. We obtain a sequence of integers $\{k_j\}_{j\geq 0}$ such that $j+1 \leq k_j \leq k_{j-1}$ for all $j$. 
    The sequence stops at the final index $j=j_\star$ when $k_{j_\star}=j_\star+1$. 
    See Figure \ref{fig:blocking} when $j_*=3$. 
    Since $\mathcal{E}_{t_{j_*}}$ is disjoint from the bubbles $\Bub^{t_{j_*-1}} \cup\Bub^{t_{j_*+1}}$, it must cross $\mathcal{E}_{t_{j_*-1},r}^-$ and thus its total width is at most one. This is a contradiction.
\end{proof}

\begin{theorem}[Amplification]
\label{thm:amplification}
    There exists a universal constant $C>0$ such that for all $n \geq 0$,
    \[
        W(\mathcal{F}_{q_{n}}) \geq \frac{4}{3} W(\mathcal{F}_{q_{n+4}}) - C.
    \]
\end{theorem}

\begin{proof}
    Fix $n\geq 0$ and consider $\mathbf{t} = \mathbf{t}(n+2) \in \{q_{n+2},\ldots,2q_{n+2}-1\}$ together with the lamination $\mathcal{G}_{\mathbf{t}}$ from Lemma \ref{lem:blocking}. 
    Since $q_{n+2} \leq t< q_{n+4}$, we have
    \begin{equation}
    \label{eqn:n-to-t}
        W(\mathcal{G}_{\mathbf{t}}) \geq W(\mathcal{F}_{q_{n+4}}) - O(1).
    \end{equation}
    according to Proposition \ref{prop:weak-improvement}.
    
    Let us identify every leaf $\gamma$ of $\mathcal{G}_{\mathbf{t}}$ as a parametrized curve $\gamma: [0,1] \to \C$ where $\gamma(0)$ is in $\Bub^{\mathbf{t}}$ and $\gamma(1)$ is a point on $\partial \D$. 
    There exists some smallest time $s_\gamma \in (0,1]$ such that $\gamma|_{(0,s_\gamma)}$ is a proper curve in the punctured disk $\Uext^{{\mathbf{t}}-q_{n}}_{q_{n}}$. 
    By Lemma \ref{lem:univalent-pushforward}, there exists a sub-lamination $\mathcal{G}^{\new}_{\mathbf{t}}$ of $\mathcal{G}_{\mathbf{t}}$ with width
    \[
         W(\mathcal{G}^{\new}_{\mathbf{t}}) \geq W(\mathcal{G}_{\mathbf{t}}) - O(1)
    \]
    such that the restriction 
    \[
    \mathcal{X}_{\mathbf{t}} := \{\gamma|_{(0,s_\gamma)} \: : \: \gamma \in \mathcal{G}^{\new}_{\mathbf{t}}\}
    \]
    can be univalently pushed forward by $\fext^{{\mathbf{t}}-q_{n}}$ to a proper lamination $\tilde{\mathcal{X}}_{\mathbf{t}}$ in $\Uext_{q_{n}}$.
    
    The critical points of $\fext^{{\mathbf{t}}-q_{n}}$ closest to the root $\crit_{-{\mathbf{t}}}$ are $\crit_{-{\mathbf{t}}+q_{n+1}}$ on the left and $\crit_{-{\mathbf{t}}+q_{n+2}}$ on the right. 
    Every leaf in the new lamination $\tilde{\mathcal{X}}_{\mathbf{t}}$ starts from the bubble $\Bub^{q_{n}}$ and ends at a point on $\partial \D \cup \partial \Bub^{q_n}$ outside of $[\crit_{q_{n+1}-q_{n}},\crit_{q_{n+2}-q_{n}}]$. 
    In particular, by Lemma \ref{lem:local-width},
    \[
        W(\mathcal{X}_{\mathbf{t}}) \leq W(\mathcal{F}_{q_{n}}) + O(1).
    \]

    Consider $l=l({\mathbf{t}})$ and $r=r({\mathbf{t}})$, and assume without loss of generality that $l<r$. 
    We will now split $\mathcal{G}^{\new}_{\mathbf{t}}$ into a disjoint union of two sub-laminations $\mathcal{S}_l$ and $\mathcal{S}_r$, where every leaf of $\mathcal{S}_l$ (resp. $\mathcal{S}_r$) intersects the bubble $\Bub^{l}$ (resp. $\Bub^r$) before it intersects $\Bub^r$ (resp. $\Bub^l$) if it ever happens. 

    We claim that the following inequality holds:
    \begin{align}
        \label{ineq:amplification}
        W(\mathcal{G}_{\mathbf{t}}) \leq \frac{3}{4} W(\mathcal{F}_{q_{n}}) + O(1).
    \end{align}
    Together with (\ref{eqn:n-to-t}), this inequality implies the desired inequality.
    To prove it, there are three cases to consider. 
    \vspace{0.05in}

\begin{figure}
    \centering
\begin{tikzpicture}
    \node[anchor=south west,inner sep=0] (image) at (0,0) {\includegraphics[width=1\linewidth]{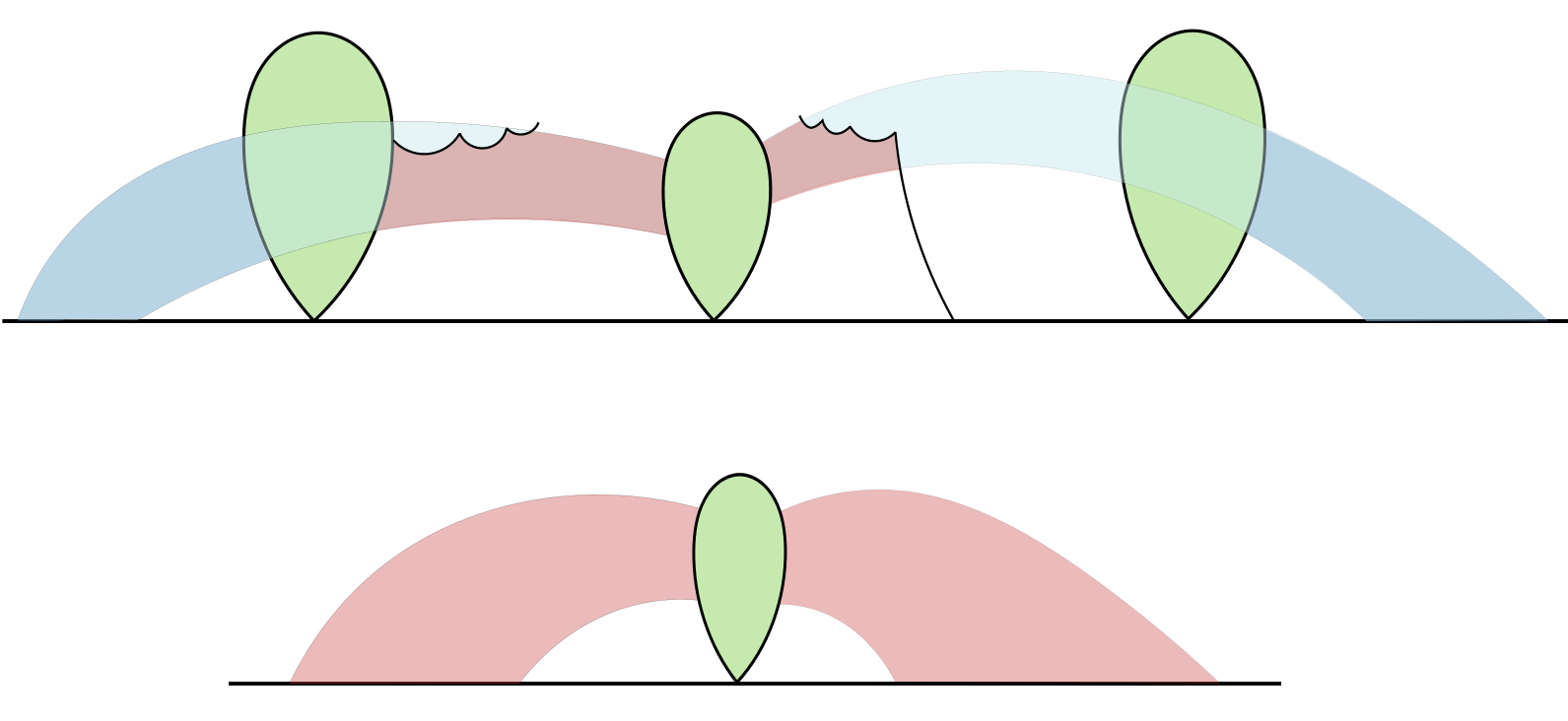}};
    \begin{scope}[
        x={(image.south east)},
        y={(image.north west)}
    ]
    \node [blue] at (0.1,0.67) {$\mathcal{Y}_{l}$};
    \node [blue] at (0.875,0.67) {$\mathcal{Y}_{r}$};
    \node [red!75!black, font=\bfseries] at (0.36,0.75) {$\mathcal{X}_{\mathbf{t}}$};
    \node [red!75!black, font=\bfseries] at (0.52,0.78) {$\mathcal{X}_{\mathbf{t}}$};
    \node [red!75!black, font=\bfseries] at (0.33,0.2) {$\tilde{\mathcal{X}}_{\mathbf{t}}$};
    \node [red!75!black, font=\bfseries] at (0.6,0.2) {$\tilde{\mathcal{X}}_{\mathbf{t}}$};
    \node [black] at (0.32,0.63) {\small $\Uext^{t-q_{n}}_{q_{n}}$};

    \node [black] at (0.455,0.555) {$\bullet$};
    \node [black] at (0.455,0.5) {$\crit_{-{\mathbf{t}}}$};
    \node [black] at (0.2,0.555) {$\bullet$};
    \node [black] at (0.2,0.5) {$\crit_{-l}$};
    \node [black] at (0.757,0.555) {$\bullet$};
    \node [black] at (0.757,0.5) {$\crit_{-r}$};
    \node [black] at (0.47,0.05) {$\bullet$};
    \node [black] at (0.47,0) {$\crit_{-q_{n}}$};

    \draw[line width=0.5pt,-latex] (0.4,0.63) -- (0.4,0.35);
    \node [black] at (0.36,0.45) {$\fext^{{\mathbf{t}}-q_{n}}$};
    \end{scope}
\end{tikzpicture}
    \caption{Illustration of Case 1 where ${\mathbf{t}}>r$ and ${\mathbf{t}}>l$}
    \label{fig:case-1}
\end{figure}
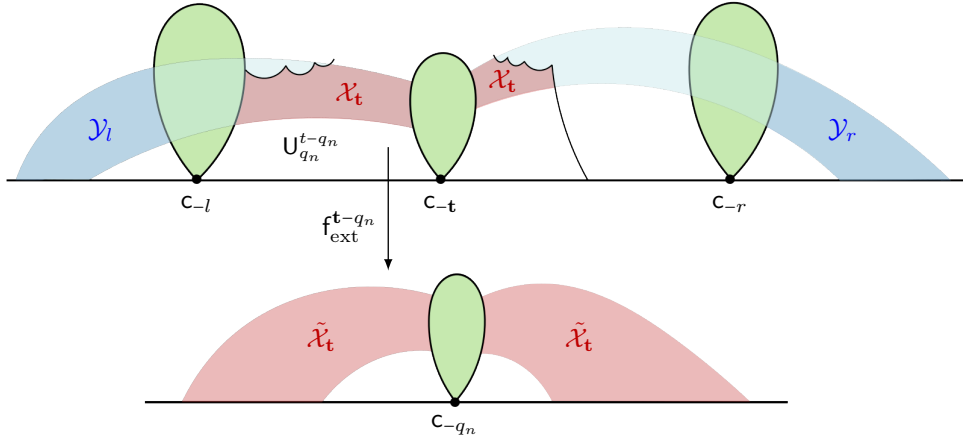

    \noindent \underline{Case 1:} ${\mathbf{t}}>r$ and ${\mathbf{t}}>l$
    \vspace{0.05in}
    
    Since $r={\mathbf{t}}-(q_{n+2}-q_{n+1}) \leq {\mathbf{t}}-q_{n}$, the bubble $\Bub^{r}$ is disjoint from the (open) disk $\Uext^{{\mathbf{t}}-q_{n}}_{q_{n}}$ as shown on the Figure~\ref{fig:case-1}.
    In particular, the last time $\nu_\gamma \in (0,1)$ where $\gamma(\nu_\gamma)$ is in $\Bub^r$ must satisfy $\nu_\gamma \geq s_{\gamma}$. 
    Consider the restriction $\mathcal{Y}_r := \{\gamma|_{(\nu_\gamma,1)} \: : \: \gamma \in \mathcal{S}_r\}$, which is disjoint from $\mathcal{X}_{\mathbf{t}}$.   Observe that $\mathcal{Y}_r$ is contained in $\mathcal{F}_r$. 
    Since $q_{n}<r<q_{n+4}$, by Proposition \ref{prop:weak-improvement},
    \[
        W(\mathcal{Y}_r) \leq W(\mathcal{F}_r) \leq W(\mathcal{F}_{q_{n}}) + O(1).
    \]
    Similarly, $\mathcal{S}_l$ also admits a restriction $\mathcal{Y}_l$ that is disjoint from $\mathcal{X}_t$ and has width
    \[
        W(\mathcal{Y}_l) \leq W(\mathcal{F}_{q_{n}}) + O(1).
    \]
    Then, $\mathcal{G}_t^{\new}$ overflows two disjoint restrictions $\mathcal{X}_{\mathbf{t}}$ and $\mathcal{Y}_l \cup \mathcal{Y}_r$. By the Series Law,
    \begin{align*}
        W(\mathcal{G}_{\mathbf{t}}) \leq W(\mathcal{X}_{\mathbf{t}}) \oplus \big[ W(\mathcal{Y}_l) + W(\mathcal{Y}_r) \big] + O(1) \leq W(\mathcal{F}_{q_{n}}) \oplus \big[ 2 W(\mathcal{F}_{q_{n}})] + O(1).
    \end{align*}  
    Since $1 \oplus 2 = \frac{3}{4}$ this inequality implies (\ref{ineq:amplification}). 
    \vspace{0.1in}

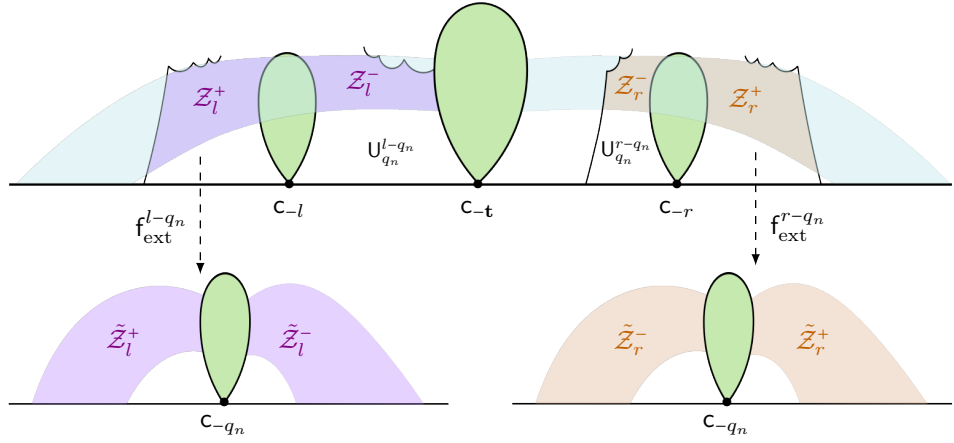
\begin{figure}
    \centering
\begin{tikzpicture}
    \node[anchor=south west,inner sep=0] (image) at (0,0) {\includegraphics[width=1\linewidth]{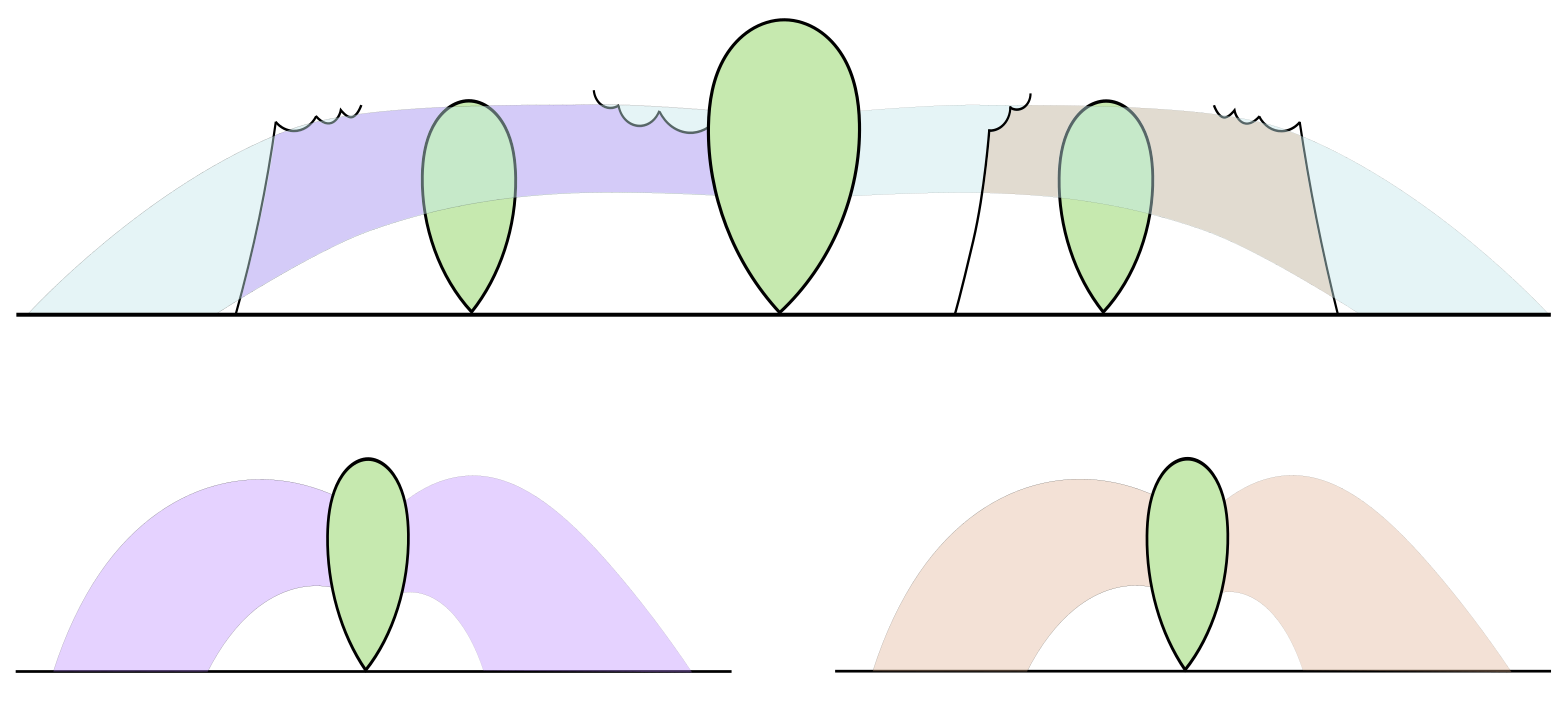}};
    \begin{scope}[
        x={(image.south east)},
        y={(image.north west)}
    ]
    
    \node [black] at (0.303,0.555) {\small $\bullet$};
    \node [black] at (0.303,0.50) {$\crit_{-l}$};
    \node [black] at (0.5,0.555) {\small $\bullet$};
    \node [black] at (0.5,0.50) {$\crit_{-{\mathbf{t}}}$};
    \node [black] at (0.708,0.555) {\small $\bullet$};
    \node [black] at (0.708,0.50) {$\crit_{-r}$};
    \node [black] at (0.235,0.055) {\small $\bullet$};
    \node [black] at (0.235,0.005) {$\crit_{-q_{n}}$};
    \node [black] at (0.76,0.055) {\small $\bullet$};
    \node [black] at (0.76,0.005) {$\crit_{-q_{n}}$};
    \node [orange!75!black] at (0.66,0.78) {$\mathcal{Z}_r^-$};
    \node [orange!75!black] at (0.78,0.75) {$\mathcal{Z}_r^+$};
    \node [violet] at (0.38,0.79) {$\mathcal{Z}_l^-$};
    \node [violet] at (0.22,0.75) {$\mathcal{Z}_l^+$};
    \node [orange!75!black] at (0.66,0.2) {$\tilde{\mathcal{Z}}_r^-$};
    \node [orange!75!black] at (0.85,0.2) {$\tilde{\mathcal{Z}}_r^+$};
    \node [violet] at (0.31,0.2) {$\tilde{\mathcal{Z}}_l^-$};
    \node [violet] at (0.13,0.2) {$\tilde{\mathcal{Z}}_l^+$};
    \node [black] at (0.41,0.63) {\scalebox{0.8}{$\Uext^{l-q_{n}}_{q_{n}}$}};
    \node [black] at (0.655,0.63) {\scalebox{0.8}{$\Uext^{r-q_{n}}_{q_{n}}$}};

    \draw[line width=0.5pt,-latex,dashed] (0.21,0.62) -- (0.21,0.35);
    \node [black] at (0.17,0.46) {$\fext^{l-q_{n}}$};

    \draw[line width=0.5pt,-latex,dashed] (0.79,0.63) -- (0.79,0.37);
    \node [black] at (0.835,0.46) {$\fext^{r-q_{n}}$};
    \end{scope}
\end{tikzpicture}
    \caption{Illustration of Case 2 where ${\mathbf{t}}<l$ and ${\mathbf{t}}<r$}
    \label{fig:case-2}
\end{figure}

    \noindent \underline{Case 2:} ${\mathbf{t}}<l$ and ${\mathbf{t}}<r$
    \vspace{0.05in}
    
    Since $r-q_{n} = {\mathbf{t}} + q_{n+1}-q_{n} > {\mathbf{t}}$, the boundary of the bubble $\Bub^{\mathbf{t}}$ is contained in the boundary of $\Uext^{r-q_{n}}_{q_{n}}$. 
    For every leaf $\gamma$ of $\mathcal{S}_r$, let $\eta^-_\gamma, \eta^+_\gamma \in (0,1)$ denote the first and the last times such that a point in $\gamma$ is contained in the bubble $\Bub^r$. 
    We will also select times $\xi^-_\gamma \in [0,\eta^-_\gamma)$ and $\xi^+_\gamma \in (\eta^+_\gamma,1]$ such that both $\gamma^- := \gamma|_{(\xi^-_\gamma,\eta^-_\gamma)}$ and $\gamma^+ :=\gamma|_{(\eta^+_\gamma,\xi^+_\gamma)}$ are proper curves in the punctured disk $\Uext^{r-q_{n}}_{q_{n}}$. 
    For $\bullet \in \{+,-\}$, consider the restrictions 
    \[
    \mathcal{Z}^\bullet_r := \{\gamma^\bullet \: : \: \gamma \in \mathcal{S}_r\}. 
    \]
    
    By Lemma \ref{lem:univalent-pushforward}, we can remove $O(1)$ buffers from $\mathcal{Z}^\bullet_r$ to get a sub-lamination that can be pushed forward by $\fext^{r-q_{n}}$ to get a proper lamination $\tilde{\mathcal{Z}}^\bullet_r$ in $\Uext_{q_{n}}$. 
    Observe that $\tilde{\mathcal{Z}}^+_r$ and $\tilde{\mathcal{Z}}^-_r$ are two disjoint laminations whose leaves start from $\partial\Bub^{q_{n}}$ and land at points on either $\partial \Bub^{q_n}$ or $\partial \D$ of combinatorial distance greater than $l_{n+1}$ away from the root $\crit_{-q_{n}}$. 
    By Lemma \ref{lem:local-width}, we can again remove $O(1)$ buffers
    from $\tilde{\mathcal{Z}}^+_r$ and $\tilde{\mathcal{Z}}^-_r$ and obtain a sub-lamination $\check{\mathcal{Z}}_r$ of $\tilde{\mathcal{Z}}^+_r \cup \tilde{\mathcal{Z}}^-_r$ that is contained in $\mathcal{F}_{q_n}$. 
    Thus, by the Series Law,
    \begin{align*}
        W(\mathcal{S}_r) \leq W(\mathcal{Z}^+_r) \oplus W(\mathcal{Z}^-_r)
        &\leq \frac{1}{4} \big[ W(\mathcal{Z}^+_r) + W(\mathcal{Z}^-_r) \big] \\
        &\leq \frac{1}{4} \big[ W(\tilde{\mathcal{Z}}^+_r) + W(\tilde{\mathcal{Z}}^-_r) \big] +O(1) \\
        &\leq \frac{1}{4} W(\mathcal{F}_{q_{n}}) + O(1).
    \end{align*}
    Similarly, for $\mathcal{S}_l$, we also have
    \[
        W(\mathcal{S}_l) \leq \frac{1}{4} W(\mathcal{F}_{q_{n}}) + O(1).
    \]
    In total, we have
    \begin{align*}
        W(\mathcal{G}_{\mathbf{t}}) &\leq W(\mathcal{S}_l) + W(\mathcal{S}_r) + O(1) \leq \frac{1}{2} W(\mathcal{F}_{q_{n}}) + O(1),
    \end{align*}  
    which is stronger than the desired (\ref{ineq:amplification}).
    \vspace{0.1in}

    \noindent \underline{Case 3:} $l < {\mathbf{t}} < r$
    \vspace{0.05in}

    This is a mixture of the first two cases. 
    The same conclusion is attained:
    \begin{align*}
        W(\mathcal{G}_{\mathbf{t}}) &\leq W(\mathcal{X}_t) \oplus W(\mathcal{Y}_l) + W(\mathcal{S}_r) + O(1) \\
        &\leq W(\mathcal{F}_{q_{n}}) \oplus W(\mathcal{F}_{q_{n}}) + \frac{1}{4} W(\mathcal{F}_{q_{n}}) + O(1) \\
        &\leq \frac{3}{4} W(\mathcal{F}_{q_{n}}) + O(1). \qedhere
    \end{align*}  
\end{proof}

\begin{proof}[Proof of Theorem \ref{thm:curve-families}]
    By Proposition \ref{prop:weak-improvement}, we just need to obtain a uniform upper bound for the width of $\mathcal{F}_{q_n}$ for all $n$.
    Since $1 < 1.2 < \frac{4}{3}$, the Amplification Theorem above can be reformulated as follows.
    There is a universal constant $K>0$ such that for $n \geq 0$,
    \begin{align}
    \label{eqn:amplify}
        W(\mathcal{F}_{q_{n+4}}) \geq K \quad \Longrightarrow \quad W(\mathcal{F}_{q_{n}}) \geq 1.2 \cdot W(\mathcal{F}_{q_{n+4}}).
    \end{align}
    We know from (\ref{eqn:W-linear}) that there exists another universal constant $K'>0$ such that 
    \begin{align}
    \label{eqn:initial-bound}
        \max_{0\leq i \leq 3} W(\mathcal{F}_{q_i}) < K'.
    \end{align}
    We will show that for all $k$, the width of $\mathcal{F}_k$ is bounded above by $\max\{K,K'\}$.
    Suppose for a contradiction that this fails for some $k \in \N$, which we can write as $k=4n+i$ for some $n \geq 1$ and $i \in \{0,1,2,3\}$. 
    By (\ref{eqn:amplify}), we have 
    \[
        W(\mathcal{F}_{q_i}) \geq 1.2^n W(\mathcal{F}_{q_k}) \geq 1.2^n \cdot \max\{K,K'\} > K',
    \]
    which is a contradiction to (\ref{eqn:initial-bound}).
\end{proof}

\begin{remark}
    One may wonder whether or not the discussion in this section also works for pseudo-bubbles. 
    While the argument in \S\ref{ss:preparation} can be repeated, an issue that arises in the amplification process is that a priori we do not know whether pseudo-bubbles $\hat{\Bub}^{q_n}$, $\ldots$, $\hat{\Bub}^{2q_n-1}$ are pairwise disjoint.
    In the next section, we will apply these bounds to obtain uniform control of pseudo-bubbles.
\end{remark}

\section{Pseudo-bubble bounds}
\label{sec:pseudo-bubble-bounds}

In this section, we will prove Theorem \ref{main-thm-01}. 
Firstly, we will discuss how to transfer the uniform bounds described in Theorem \ref{thm:curve-families} to obtain angular and size control of pseudo-bubbles of $f = f_\theta$ for a fixed $\theta \in \Theta_{\textnormal{bdd}}$ (Theorem \ref{thm:pseudo-bubble-bounds}). 
Secondly, we will use nests of tilings to show that pseudo-bubbles must be uniform quasidisks (Theorem \ref{thm:uniform-quasiconformality}).

\subsection{Uniform size and angle}
\label{ss:uniform-size}
Let us use the notation $\mathsf{C}_k(K)$ introduced in \S\ref{ss:bubbles-and-pseudo-bubbles}. 
The next theorem states that we have uniform angular and size control in external coordinates.

\begin{theorem}
\label{thm:size-external-pseudo-bubbles}
    There exists a universal constant $K>1$ such that for every $k,n \in \N$ with $q_n \leq k < q_{n+1}$, the pseudo-bubble $\hat{\Bub}^k$ of $\fext$ rooted at $\crit_{-k}$ is contained in the cone $\mathsf{C}_k(K)$ and it has diameter 
        \[
        \frac{1}{K}|\crit_{-k} - \crit_{-k+q_{n}}| \leq \textnormal{diam}(\hat{\Bub}^k) \leq K |\crit_{-k} - \crit_{-k+q_{n}}|.
        \]
\end{theorem}

\begin{proof}
    Fix integers $k,n \in \N$ as stated above.
    Lemma \ref{lem:control-near-roots} and Theorem \ref{thm:curve-families} imply that the bubble $\Bub^k$ lies in some cone $\mathsf{C}_k(K')$ for some universal $K'>0$ and that it has diameter comparable to $|\crit_{-k}-\crit_{-k+q_n}|$. Recall that the boundary of $\hat{\Bub}^k$ consists of parts of the boundary of $\Bub^k$ together with a collection of dams. In order to prove the theorem, it then suffices to show that every dam on the boundary of $\hat{\Bub}^k$ has uniformly bounded hyperbolic diameter in $\C \backslash \overline{\D}$.
    
    Consider any dam $\beta$ on the boundary of the pseudo-bubble $\hat{\Bub}^k$ of $\fext$. 
    It is the lift of a dam $\beta_0$ of the main pseudo-bubble $\hat{\Bub}^0$ attached to the critical point. 
    By construction, each dam comes with a collar, so there exists a small disk neighborhood $\Next_0$ of $\beta_0$ disjoint from $\overline{\D}$ such that the annulus $\Next_0 \backslash \beta_0$ has modulus $ \succ 1$. 
    The disk $\Next_0$ univalently lifts under $\fext^k$ to a disk neighborhood $\Next$ of $\beta$ disjoint from $\overline{\D}$. 
    Since $\Next \backslash \beta$ has definite modulus, every dam $\beta$ clearly has uniformly bounded hyperbolic diameter in $\C \backslash \overline{\D}$.
\end{proof}

To transfer to the dynamical plane of $f$, we need the following lemma. 
For $k \in \N$ and $K>0$, denote 
\[
C_k(K) := \psi^{-1}(\mathsf{C}_k(K)),
\]
the corresponding cone in the dynamical plane of $f$.
Recall the combinatorial threshold $\Mbold$ in the construction of the pseudo-Siegel disk $\hat{Z} = \hat{Z}(\Mbold)$ described in \S\ref{ss:pseudo-siegel}. 

\begin{lemma}
\label{lem:disjointness}
    The combinatorial threshold $\Mbold$ of the pseudo-Siegel disk $\hat{Z} = \hat{Z}(\Mbold)$ can be chosen to be sufficiently high such that the following additional property is satisfied. 
    There exists a universal constant $K>1$ such that for all $n \geq 0$,
    \begin{enumerate}[label=\textnormal{(\arabic*)}]
        \item for every $k \in \{q_n,\ldots,q_{n+1}-1\}$, the cone $C_{k}(K)$ is disjoint from $\hat{Z}^n$;
        \item the cone $C_{q_n}(K)$ is also disjoint from $\hat{Z}$.
    \end{enumerate}
\end{lemma}

\begin{proof}
    The proof below is done in external coordinates for convenience.
    Recall the nest of tilings $\mathcal{D}_m$ of $\partial \D$ described in \S\ref{ss:real-bounds}.
    Fix a sufficiently small constant $\boldsymbol{\epsilon} > 0$. 
    We can select $\Mbold$ to be sufficiently high enough (depending on $\boldsymbol{\epsilon}$) such that in the dynamical plane of $\fext$, whenever $a_{m+2}>\Mbold$ for some $m \geq -1$, the dam $\beta_I$ attached to any tile $I$ in $\mathcal{D}_m$ satisfies
    \[
        \text{diam}(\beta_I) < \varepsilon \cdot \text{dist}(\beta_I, \partial I).
    \]
    This is possible by virtue of Theorem \ref{thm:R-APB} and the fact that dams are hyperbolic geodesics.
    
    Let us fix integers $n\geq 0$, $k \in \{q_n, \ldots,q_{n+1}-1\}$, and $m \geq -1$ such that $a_{m+2} > \Mbold$. 
    There are two cases.
    \begin{itemize}
        \item Suppose $m \geq n$. 
        Then, the pre-critical point $\crit_{-k}$ is an endpoint of some tile in $\mathcal{D}_m$. 
        Hence, for any point $x$ in any dam of level $m$, the distance between $x$ and $\partial \D$ is less than $\boldsymbol{\epsilon} |x-\crit_{-k}|$.
        \item Suppose $m < n$. 
        The pre-critical point $\crit_{-q_n}$ is contained in either one of the two tiles $[\crit_{-q_m},\crit_0]$ or $[\crit_0,\crit_{-q_{m+1}+q_{m}}]$ of $\mathcal{D}_m$. 
        Since $n > m$, Theorem \ref{thm:R-APB} implies that the distance between $\crit_0$ and $\crit_{-q_n}$ is very small compared to the distance between $\crit_0$ and any dam of level $m$. 
        More precisely, there is some small constant $\boldsymbol{\epsilon}'=\boldsymbol{\epsilon}'( \boldsymbol{\epsilon})>0$ such that for any point $x$ on any dam of level $m$, the distance between $x$ and $\partial \D$ is still less than $\boldsymbol{\epsilon}' |x-\crit_{-q_n}|$. 
    \end{itemize}
    The first observation implies that there exists some sufficiently small $K = K(\boldsymbol{\epsilon}) > 0$ such that dams of level $\geq n$ are disjoint from $\mathsf{C}_k(K)$.
    Together with the second observation, we also have that there exists some sufficiently small $K = K(\boldsymbol{\epsilon}') > 0$ such that dams of level $\geq n$ are disjoint from $\mathsf{C}_{q_n}(K)$.
    These imply the desired properties (1) and (2).
\end{proof}

Theorem \ref{thm:size-external-pseudo-bubbles} and Lemma \ref{lem:disjointness} (2) imply the following.
Consider the map $\psi_{\textnormal{qc}}$, a modification of $\psi$ coming from Corollary \ref{cor:new-psi}.

\begin{corollary}
\label{cor:uniform-qc-transfer}
    The uniformly quasiconformal map $\psi_{\textnormal{qc}}$ sends $\hat{B}^{q_n}$ onto $\hat{\Bub}^{q_n}$ for all $n \geq 0$.
\end{corollary}

Now, we are ready to prove parts (2) and (3) of Theorem \ref{main-thm-01}.

\begin{theorem}
\label{thm:pseudo-bubble-bounds}
    There exists a universal $K>1$ such that 
    \begin{enumerate}[label=\textnormal{(\arabic*)}]
        \item every primary pseudo-bubble $\hat{B}^{k}$ of $f$ is contained in the cone $C_k(K)$ and satisfies
        \[
            \frac{1}{K} |c_{-k} - c_{-k+q_{n}}| \leq \textnormal{diam}(\hat{B}^{k}) \leq K |c_{-k} - c_{-k+q_{n}}|,
        \]
        where $n \in \N$ is the unique integer such that $q_n \leq k < q_{n+1}$;
        \item with respect to the hyperbolic metric of $\C \backslash \overline{Z}$, the diameter of any non-primary pseudo-bubble $\hat{B}$ of $f$ is at most $K$.
    \end{enumerate}
\end{theorem}

\begin{proof}
    Throughout, $K$ will represent a sufficiently large universal positive constant.  
    
    Item (1) for $k=q_n$ follows directly from Theorem \ref{thm:size-external-pseudo-bubbles} and Corollary \ref{cor:uniform-qc-transfer}.
    As a result, we can construct an open disk neighborhood $E^{q_n}$ of $\hat{B}^{q_n}$ such that 
    \begin{itemize}
        \item[(i)] the intersection of $E^{q_n}$ and the boundary of $\hat{Z}^{n}$ is the interval with endpoints $c_{-q_n+q_{n+1}}$ and $c_{-q_{n+1}}$;
        \item[(ii)] if we denote by $J_n$ the closed interval on $\partial \hat{Z}$ with endpoints $c_{-q_n}$ and $c_{0}$, then $E^{q_n}\backslash (\hat{B}^{q_n} \cup J_n)$ is an annulus with modulus $\geq K^{-1}$.
    \end{itemize}
    Fix an integer $k$ with $q_n < k < q_{n+1}$. 
    Since $E^{q_n}$ contains no critical values of $f^{q_{n+1}-q_n}$, the inverse branch $g$ of $f^{k-q_n}$ that sends $\hat{B}^{q_n}$ to $\hat{B}^k$ is well-defined on $E^{q_n}$. 
    By property (ii), $g$ has uniformly bounded distortion on $\hat{B}^{q_n} \cup J_n$. 
    This implies item (1) for all $k$.

    Next, similar to the disks $E^{q_n}$ described above, for every $k,n \geq 0$ with $q_n \leq k < q_{n+1}$, item (1) guarantees the existence of a disk neighborhood $U^k$ of $\hat{B}^k$ such that 
    \begin{itemize}
        \item[(iii)] the intersection of $U^{k}$ and $\partial \hat{Z}^{n}$ is the interval with endpoints $c_{-k+q_{n+1}}$ and $c_{-k+q_n-q_{n+1}}$;
        \item[(iv)] $\textnormal{mod}(U^k\backslash\hat{B}^{k}) \geq K^{-1}$.
    \end{itemize}
    Denote by $\hat{B}^{k,0}$ the lift of $\hat{B}^k$ under $f$ that is attached to $\hat{B}^0$.
    Since $U^k$ does not contain the critical value of $f$, then $U^k$ univalently lifts to a disk neighborhood $U^{k,0}$ of $\hat{B}^{k,0}$. 
    Property (iii) implies that $U^{k,0}$ is disjoint from $\overline{Z}$ and property (iv) implies that $U^{k,0} \backslash \hat{B}^{k,0}$ has definite modulus. 
    Hence, the hyperbolic diameter of $\hat{B}^{k,0}$ is bounded above by some universal constant independent of $k$. 
    By further pulling back $\hat{B}^{k,0}$ under $f$ and applying Schwarz Lemma with respect to the hyperbolic metric of the complement of $\overline{Z}$, item (2) follows.
\end{proof}

\subsection{Uniform quasiconformality}
\label{ss:uniform-qc}

Recall from \S\ref{ss:pseudo-siegel} the existence of a nest of tilings $\mathcal{T}_m(\hat{Z})$, $m\geq -1$ on the boundary of the pseudo-Siegel disk $\hat{Z}$. Given any pseudo-bubble $\hat{B}$ of $f$ of some generation $g$, we can lift $\mathcal{T}_m\left(\hat{Z}\right)$, $m\geq -1$ to a nest of tilings $\mathcal{T}_m\left(\hat{B}\right)$, $m\geq -1$ on the boundary of $\hat{B}$ via an appropriate inverse branch of $f^{g+1}$. 
Similarly, given any pseudo-bubble $\hat{\Bub}$ in external coordinates, we can define the nest of tilings $\mathcal{T}_m\left(\hat{\Bub}\right)$, $m \geq -1$ on the boundary of $\hat{\Bub}$ by transferring the nest of tilings of $\psi^{-1}\left(\hat{\Bub}\right)$ via $\psi$. 
Trivially, $\mathcal{T}_m\left(\hat{\Bub}\right)$ always has uniformly bounded combinatorics and inner geometry. However, it is unclear whether or not $\mathcal{T}_m\left(\hat{\Bub}\right)$ has uniformly bounded outer geometry.

First, let us consider a principal pseudo-bubble $\hat{\Bub}^{q_n}$ of $\fext$ for some $n \geq 3$.
We will define a new nest of tilings 
\[
\widetilde{\mathcal{T}}_k = \widetilde{\mathcal{T}}_k \left(\hat{\Bub}^{q_n} \right), \quad k \geq 0
\]
by modifying $\mathcal{T}_k = \mathcal{T}_k \left(\hat{\Bub}^{q_n}\right)$ appropriately.
We will show in the next lemma that this new nest of tilings has uniformly bounded combinatorics and inner and outer geometry.

For $m \in \{3, \ldots, n\}$, let
\begin{itemize}
    \item $X^*_m$ be the tile in $\mathcal{T}_m$ that contains the root $\crit_{-q_n}$, 
    \item $X^l_m$ and $X^r_m$ be the two distinct tiles in $\mathcal{T}_m$ that are adjacent to $X^*_m$, 
    \item $\mathbf{X}_m $ be the union $X_m^l \cup X_m^* \cup X_m^r$, and
    \item $J^{n-m}$ be the closure of $\partial \hat{\Bub}^{q_n} \backslash \mathbf{X}_m$.
\end{itemize}
Note that 
\[
    \mathbf{X}_3 \Supset \mathbf{X}_4 \Supset \ldots \Supset \mathbf{X}_n 
    \quad \text{ and } \quad
    J^{n-3} \Subset J^{n-4} \Subset \ldots \Subset J^0.
\]

We set the tiling $\widetilde{\mathcal{T}}_{0}$ to consist of $J^0$, which we call the unique \emph{extraordinary} tile, and all the tiles in $\mathcal{T}_{n+1}$ that are contained in $\mathbf{X}_n$, which we call \emph{ordinary} tiles.
Next, we inductively define the tiling $\widetilde{\mathcal{T}}_{k}$ for $k \geq 1$ by the following set of rules. 
\begin{enumerate}[label=\textnormal{(\roman*)}]
    \item Every ordinary tile $I$ in $\widetilde{\mathcal{T}}_{k-1}$ decomposes into ordinary tiles in $\widetilde{\mathcal{T}}_k$ as follows.
    If $I$ belongs to $\mathcal{T}_m$ for some $m$, then every tile in $\mathcal{T}_{m+1}$ that is contained in $I$ belongs in $\widetilde{\mathcal{T}}_{k}$.
    \item For $k \leq n-3$, $\widetilde{\mathcal{T}}_k$ has a unique extraordinary tile, which is $J^k$.
    \item For $k \leq n-3$, every tile of $\mathcal{T}_{n+1-k}$ that is contained in $J^{k-1} \backslash J^k$ is an ordinary tile in $\widetilde{\mathcal{T}}_k$.
    \item Every tile of $\mathcal{T}_{4}$ that is contained in $J^{n-3}$ is an ordinary tile in $\widetilde{\mathcal{T}}_{n-2}$. 
\end{enumerate} 
Observe that for $k \geq n-2$, every tile of $\widetilde{\mathcal{T}}_k$ is deemed ordinary.

\begin{lemma}
\label{lem:uniform-tiling}
    For $n \geq 3$, the modified nest of tilings $\widetilde{\mathcal{T}}_k(\hat{\Bub}^{q_n})$, $k\geq 0$ has uniformly bounded combinatorics and inner and outer geometry.
\end{lemma}

\begin{proof}
    Fix $n$ and write $\mathcal{T}_m = \mathcal{T}_m \left(\hat{\Bub}^{q_n} \right)$ for $m \geq -1$ and $\widetilde{\mathcal{T}}_k = \widetilde{\mathcal{T}}_k \left(\hat{\Bub}^{q_n} \right)$ for $k \geq 0$. 
    It is immediate from the construction that the nest $\widetilde{\mathcal{T}}_k$, $k\geq 0$ has uniformly bounded combinatorics.

    Let us fix $k \geq 0$.
    Pick any triplet of consecutive tiles $I_{-1}, I_0, I_1$ in $\widetilde{\mathcal{T}}_k$. 
    Denote $J := \partial \hat{\Bub}^{q_n} \backslash (I_{-1} \cup I_0 \cup I_1)$. 
    Denote 
    \begin{itemize}
        \item $\widetilde{\mathcal{F}}^+(I_0) = $ the family of proper curves in $\RS \backslash \hat{\Bub}^{q_n}$ that start from $I_0$ and end at $J$;
        \item $\widetilde{\mathcal{F}}^-(I_0) = $ the family of proper curves in the interior of $\hat{\Bub}^{q_n}$ that start from $I_0$ and end at $J$.
    \end{itemize}
    Our goal is to show that the width $\widetilde{W}^\pm(I_0)$ of $\widetilde{\mathcal{F}}^\pm(I_0)$ is uniformly bounded. 
    
    To control the inner geometry, we remind ourselves that the interior of $\hat{\Bub}^{q_n}$ is univalently mapped onto $\hat{\Bub}^{0}$. This implies that $\widetilde{W}^-(I_0)$ is equal to the width of proper curves in the interior of $\hat{\Bub}^{0}$ that start from $\fext^{q_n}(I_0)$ and end at $\fext^{q_n}(J)$.
    There are three distinct cases to consider.
    \begin{itemize}
        \item If $I_{-1}$, $I_0$, and $I_1$ are all ordinary, there exists some $m$ such that $I_{-1}$, $I_0$, $I_1$ are contained in $\mathcal{T}_m  \cup \mathcal{T}_{m+1} \cup \mathcal{T}_{m+2}$. 
        \item If $I_0$ is extraordinary, then both $I_{-1}$ and $I_1$ must be in $\mathcal{T}_{n+1-k}$ and the number of distinct tiles of $\mathcal{T}_{n+1-k}$ in the closure of the complement of $I_{-1} \cup I_0 \cup I_1$ is bounded above by some uniform constant.
        \item If $I_{\pm 1}$ is extraordinary, then $I_{\pm 1}$ contains a tile in $\mathcal{T}_{n+1-k}$, $I_0$ is a tile in $\mathcal{T}_{n+1-k}$, $I_{\mp 1}$ is a tile in either $\mathcal{T}_{n+1-k}$ or $\mathcal{T}_{n+3-k}$.
     \end{itemize}
     These observations, together with the fact that $\mathcal{T}_m $, $m \geq -1$ has uniformly bounded inner geometry, imply that $\widetilde{W}^-(I_0)$ is uniformly bounded from above.
     Hence, $\widetilde{T}_k$ has uniformly bounded inner geometry.
     
     An observation similar to the previous paragraph implies that the nest of tilings $\fext^{q_n}(\widetilde{\mathcal{T}}_k)$, $k\geq 0$ on $\hat{\Bub}^{0}$ has uniformly bounded outer geometry.
     For brevity, we will proceed by denoting $X^l = X_n^l$, $X^* = X_n^*$, $X^r = X_n^r$.
     To analyze the outer geometry, there are again three cases to consider.

    \vspace{0.05in}

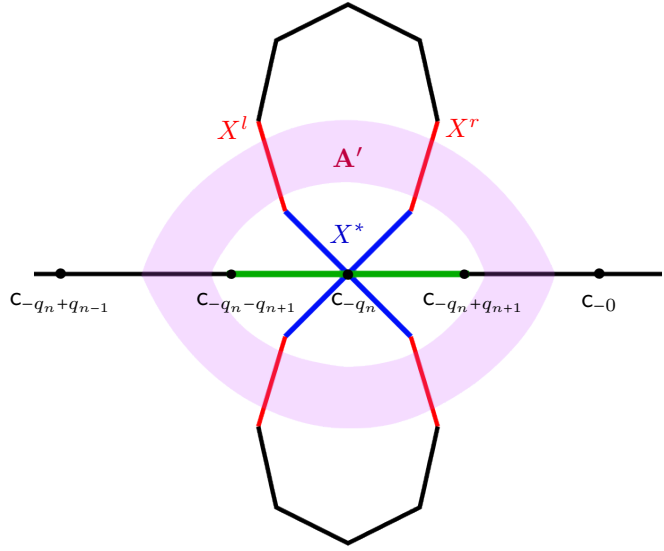
\begin{figure}
    \centering
\begin{tikzpicture}
    \node[anchor=south west,inner sep=0] (image) at (0,0) {\includegraphics[width=0.75\linewidth]{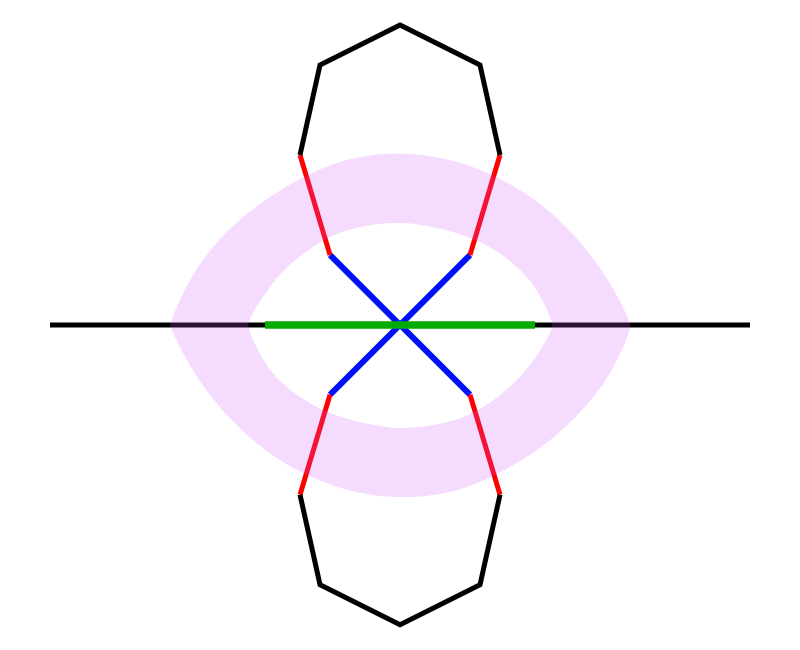}};
    \begin{scope}[
        x={(image.south east)},
        y={(image.north west)}
    ]
    
    \node [black] at (0.338,0.498) {$\bullet$};
    \node [black] at (0.36,0.45) {$\crit_{-q_n -q_{n+1}}$};
    \node [black] at (0.5,0.497) {\large $\bullet$};
    \node [black] at (0.51,0.45) {$\crit_{-q_n}$};
    \node [black] at (0.662,0.498) {$\bullet$};
    \node [black] at (0.675,0.45) {$\crit_{-q_n + q_{n+1}}$};
    \node [black] at (0.85,0.498) {\large $\bullet$};
    \node [black] at (0.85,0.45) {$\crit_{-0}$};
    \node [black] at (0.10,0.498) {\large $\bullet$};
    \node [black] at (0.10,0.45) {$\crit_{-q_n+q_{n-1}}$};
    \node [blue!80!black] at (0.5,0.57) {$X^*$};
    \node [red] at (0.34,0.75) {$X^l$};
    \node [red] at (0.66,0.75) {$X^r$};
    \node [purple] at (0.5,0.7) {$\mathbf{A}'$};
    \end{scope}
\end{tikzpicture}
    \caption{The annulus $\mathbf{A}'$ in Case 1}
    \label{fig:case-1-annulus}
\end{figure}
     
    \noindent \underline{Case 1:} $I_0$ is contained in $X$.
    \vspace{0.05in}

    By Koebe distortion theorem and Theorem \ref{thm:R-APB}, there exists an open $\partial \D$-symmetric annulus $\mathbf{A}$ with the following properties.
     \begin{enumerate}[label=\textnormal{(\alph*)}]
         \item $\mathbf{A} \cap \partial \hat{\Bub}^{0}$ is contained in the union of intervals $\fext^{q_n}(X^l)$ and $\fext^{q_n}(X^r)$.
         \item $\text{mod}(\mathbf{A}) \asymp 1$.
         \item the bounded connected component $\mathbf{D}$ of $\mathbf{A}$ contains $\fext^{q_n}(X^*)$ and the interval $[\crit_{-q_{n+1}}, \crit_{q_{n+1}}] \subset \partial \D$ together with all the dams attached to this interval.
         \item $\mathbf{A} \cup \mathbf{D}$ is disjoint from $\{\crit_i\}_{1\leq i \leq q_n}$.
     \end{enumerate}
    Property (d) implies that both $\mathbf{A}$ and $\mathbf{D}$ univalently lift under $\fext^{q_n}$ to $\partial \D$-symmetric open annulus $\mathbf{A}'$ and closed set $\mathbf{D}'$ respectively. 
    See Figure \ref{fig:case-1-annulus}.
    
    Properties (b) and (c) as well as the fact that $I_0$ is contained in $X^*$ imply that the width of curves in $\widetilde{\mathcal{F}}^+(I_0)$ that intersect the outer boundary of $\mathbf{A}'$ must be uniformly bounded above, 
    Meanwhile, the set of curves in $\widetilde{\mathcal{F}}^+(I_0)$ that stay inside of $\mathbf{A}' \cup \mathbf{D}'$ will be mapped univalently to a proper curve in $\RS \backslash \hat{\Bub}^0$ connecting $\fext^{q_n}(I_0)$ and $\fext^{q_n}(J)$. 
    Therefore, since $\fext^{q_n}(\widetilde{\mathcal{T}}_k)$, $k\geq 0$ has uniformly outer bounded geometry, then $\widetilde{\mathcal{W}}^+(I_0) = O(1)$.
    \vspace{0.05in}
     
    \noindent \underline{Case 2:} $I_0$ is an ordinary tile disjoint from $X$.
    \vspace{0.05in}

    Let us split $\widetilde{\mathcal{F}}^+(I_0)$ into a disjoint union $\mathcal{G}_1 \cup \mathcal{G}_2$ where $\mathcal{G}_1$ consists of curves that intersect $\partial \D$ and $\mathcal{G}_2$ consists of curves that avoid $\partial \D$. 
    Up to removal of $O(1)$ width, we will assume that both $\mathcal{G}_1$ and $\mathcal{G}_2$ are laminations.
    
    The width of $\mathcal{G}_1$ is bounded above by the width $W_1$ of curves connecting $\partial{\hat{\Bub}}^{q_n} \backslash X^*$ and $\partial \D$. 
    By design, $X^*$ is the unique tile of $\mathcal{T}_n$ containing $\crit_{-q_n}$ and, by Koebe control (e.g. properties (a)--(d)), its endpoints are of distance $\asymp |\crit_{-q_n}-\crit_0|$ away from $\partial \D$.
    Together with Theorem \ref{thm:size-external-pseudo-bubbles}, this implies that $W_1$ is bounded above by a universal constant.
    Hence, $W(\mathcal{G}_1) = O(1)$.

    Suppose $I_0$ is in $\mathcal{T}_m$ for some $m \geq -1$. 
    Let $I^l$ and $I^r$ denote the tiles in $\mathcal{T}_{m+1}$ that are directly on the left and right of $I_0$ respectively. 
    The union $I^l \cup I^r$ is contained in $I_{-1} \cup I_1$ and its interior is disjoint from $\partial \D$.
    Denote 
    \[
    \hat{\Uext}_0 := \C \backslash \left(\overline{\D} \cup \hat{\Bub}^0\right)
    \]
    and let $\hat{\Uext}^{q_n}_0$ be the unbounded connected component of $\fext^{-q_n} \left( \hat{\Uext}_0 \right)$. 
    Let $\mathcal{G}_2^{\new}$ denote the restriction of $\mathcal{G}_2$ in $\hat{\Uext}_0^{q_n}$ consisting of curves starting from $I_0$ and ending at $\partial \hat{\Bub}^{q_n} \cup \partial \Uext_0^{q_n}$. 
    By Lemma \ref{lem:univalent-pushforward}, there exists a subfamily $\mathcal{G}_2^{\text{New}}$ of $\mathcal{G}^{\new}$ with width
        \[
            W \left(\mathcal{G}_2^{\text{New}} \right) \geq W \left(\mathcal{G}_2^{\new}\right) - O(1)
        \]
    such that $\mathcal{G}_2^{\text{New}}$ can be univalently pushed forward by $\fext^{q_n}$. 
    Every curve in the image $\fext^{q_n}\left(\mathcal{G}_2^{\text{New}}\right)$ is a proper curve in $\hat{\Uext}_0$ starting from $\fext^{q_n}(I_0)$, a tile in $\mathcal{T}_m \left(\hat{\Bub}^0 \right)$, 
    and it protects either $\fext^{q_n} \left(I^l\right)$ or $\fext^{q_n} \left(I^r\right)$, both of which are next to $\fext^{q_n}(I_0)$ and contained in $\mathcal{T}_{m+1}\left(\hat{\Bub}^0\right)$. 
    We then conclude that 
    \[
    W \left( \fext^{q_n}(\mathcal{G}_2^{\text{New}})\right) = O(1)
    \]
    because of Lemma \ref{lem:main-bubble} and the uniformly bounded outer geometry of $\mathcal{T}_{m}\left(\hat{\Bub}^0\right)$, $m \geq -1$. Consequently, we have $W(\mathcal{G}_2) = O(1)$.    
    \vspace{0.05in}
     
    \noindent \underline{Case 3:} $I_0$ is an extraordinary tile disjoint from $X$. 
    \vspace{0.05in}
    
    There exists some $k \leq n-1$ such that $I_{-1}$ and $I_1$ are intervals in $\mathcal{T}_{n-k+1}$ and the complement of $I_0$ is the union of three intervals in $\mathcal{T}_{n-k} $. 
    In particular, we can express $J$ as a union of intervals $J_1, \ldots, J_t$ in $\mathcal{T}_{n-k+1} $ where $t$ is uniformly bounded by some universal positive constant. 
    It then suffices to show that the width of curves in $\widetilde{\mathcal{F}}^+(I_0)$ landing at $J_i$ is $O(1)$ for each $i$. 
    If $J_i$ is contained in $X$, this follows from the argument in Case 1. Otherwise, we can adapt the same argument presented in Case 2, and we are done.
\end{proof}

\begin{theorem}
\label{thm:uniform-quasiconformality}
    All pseudo-bubbles of $f$ and $\fext$ are uniform quasidisks.
\end{theorem}

\begin{proof}
    When $n \leq 2$, the map $\fext^{q_n}: \hat{\Bub}^{q_n} \to \hat{\Bub}^0$ has uniformly bounded distortion and so $\hat{\Bub}^{q_n}$ is a uniformly quasidisk.
    By applying the general criterion described in \cite[Lemma 11.3]{DL22}, Lemma \ref{lem:uniform-tiling} directly implies the uniform quasiconformality of principal pseudo-bubbles $\hat{\Bub}^{q_n}$ of $f_\ext$ for all $n$.
    By Corollary \ref{cor:uniform-qc-transfer}, the principal pseudo-bubbles $\hat{B}^{q_n}$ of $f$ are also uniform quasidisks.
    To see how we spread this property to all other pseudo-bubbles of $f$ of arbitrary generation, we can lift $\hat{B}^{q_n}$ with controlled distortion just like in the proof of Theorem \ref{thm:pseudo-bubble-bounds}. 
    The same can be done for pseudo-bubbles in external coordinates via Theorem \ref{thm:size-external-pseudo-bubbles}.
\end{proof}

This completes the proof of part (1) of Theorem \ref{main-thm-01}.

\appendix

\section{Extremal width}
\label{sec:appendix}

Given a family $\mathcal{G}$ of curves on a Riemann surface $S$, we denote by $W(\mathcal{G})$ the extremal width of $\mathcal{G}$. We list without proof a number of fundamental results on extremal width. (See \cite{A06} and the appendix in \cite{KL09} for details.)

\begin{proposition}[Parallel Law]
    \label{prop:parallel-law}
    For any two curve families $\mathcal{G}_1$ and $\mathcal{G}_2$,
    \[
    W(\mathcal{G}_1 \cup \mathcal{G}_2) \leq W(\mathcal{G}_1) + W(\mathcal{G}_2).
    \]
    Equality is achieved when $\mathcal{G}_1$ and $\mathcal{G}_2$ have disjoint support.
\end{proposition}

We say that a curve family $\mathcal{G}$ \emph{overflows} another curve family $\mathcal{H}$, denoted by $\mathcal{H} < \mathcal{G}$, if every curve in $\mathcal{G}$ contains a curve in $\mathcal{H}$ (curves in $\mathcal{G}$ are longer and fewer). We also say that $\mathcal{H}$ is a \emph{restriction} of $\mathcal{G}$ if $\mathcal{G}$ overflows $\mathcal{H}$ but not any proper subfamily of $\mathcal{H}$ (curves in $\mathcal{G}$ are longer, but not more nor fewer).

Denote by $x \oplus y$ the harmonic sum $(x^{-1}+y^{-1})^{-1}$.

\begin{proposition}[Series Law]
    \label{prop:series-law}
    Suppose a curve family $\mathcal{G}$ overflows two disjoint curve families $\mathcal{G}_1$ and $\mathcal{G}_2$. Then,
    \[
    W(\mathcal{G}) \leq W(\mathcal{G}_1) \oplus W(\mathcal{G}_2).
    \]
\end{proposition}

The following elementary inequality allows us to convert a harmonic sum into the usual sum:
\[
    \bigoplus_{i=1}^n w_i \leq \frac{1}{n^2} \sum_{i=1}^n w_i.
\]

Extremal width is invariant under conformal maps. More generally, we have the following transformation rule.

\begin{proposition}
    \label{prop:transformation-law}
    Let $f: U\to V$ be a holomorphic map between two Riemann surfaces and $\mathcal{G}$ be a family of curves in $U$. Then,
    \[
    W(f(\mathcal{G})) \leq W(\mathcal{G}).
    \]
    If $f$ is at most $d$ to $1$, then
    \[
    W(\mathcal{G}) \leq d \cdot W(f(\mathcal{G})).
    \]
\end{proposition}

A \emph{(conformal) rectangle} $P$ on a surface $S$ is the image of a continuous map $\phi: [0,m]\times[0,1] \to \overline{S}$ that restricts to a conformal embedding in the interior. 
The vertical sides of a rectangle $P$ are $\phi(\{0\}\times [0,1])$ and $\phi(\{m\} \times [0,1])$, and the horizontal sides of $P$ are $\phi([0,m]\times \{0\})$ and $\phi([0,m]\times\{1\})$. 
A curve in $P$ is called \emph{vertical} if it connects the two horizontal sides of $P$. The \emph{vertical foliation} of $P$ is defined to be the collection of curves
\[
\mathcal{F}(P) := \{ \phi\left( \{t\} \times (0,1) \right) \: | \: t \in (0,m)\}.
\]
The \emph{width} of $P$ is 
\[
W(P) := W(\mathcal{F}(P))= m.
\]
We say that $P$ \emph{crosses} a curve $\gamma$ if every vertical curve in $P$ intersects $\gamma$.

\begin{proposition}[Non-Crossing Principle]
\label{prop:non-crossing-principle}
    If a pair of rectangles $P_1$ and $P_2$ on $S$ has width $W(P_1), W(P_2)> 1$, then they never cross, i.e., there exist disjoint leaves $\gamma_1 \in \mathcal{F}(P_1)$ and $\gamma_2 \in \mathcal{F}(P_2)$.
\end{proposition}

Let $P$ be a rectangle and let $\phi: [0,m] \times [0,1] \to \overline{S}$ be the defining function as described above.
For $\varepsilon \in [0,m)$, an $\varepsilon$-\emph{buffer} of a rectangle $P$ is a sub-rectangle of $P$ of the form $\phi( [0,\varepsilon)\times (0,1)$ or $\phi( (m-\varepsilon,m] \times (0,1))$. 
A direct consequence of the non-crossing principle is the following proposition.

\begin{proposition}[{\cite[Lemma 2.14]{KL09}}]
\label{prop:nice-buffers}
    Every pair of rectangles $P_1$ and $P_2$ of width greater than $8$ admits sub-rectangles $P_1^{\new}$ and $P_2^{\new}$ obtained by removing some buffers of width $\leq 4$ such that $P_1^{\new}$ and $P_2^{\new}$ have disjoint vertical sides.
\end{proposition}

When $S$ has boundary, we say that a curve $\gamma : (0,1) \to S$ is \emph{proper} if it has well-defined endpoints $\gamma(0)$ and $\gamma(1)$ contained in $\partial S$. For any disjoint subsets $I$ and $J$ of $\partial S$, we denote by $\mathcal{F}_S (I,J)$ and $W_S(I,J)$ the family of proper curves in $S$ that connect $I$ and $J$, and its width respectively. 

When $S$ is the unit disk, the width $W_{\D} (I,J)$ can be estimated as follows. For any interval $I \subset \partial \D$, denote by $|I|$ the Euclidean length of $I$ normalized so that $| \partial \D| = 1$.

\begin{proposition}[Log-Rule, {\cite[Lemma 2.5]{DL22}}]
\label{prop:log-rule}
    Suppose $\partial \D$ is partitioned into four intervals $I_1, I_2, I_3, I_4$ labelled cyclically.
    \begin{enumerate}[label=\textnormal{(\arabic*)}]
        \item If $\min(|I_1|, |I_3|) \geq \min(|I_2|, |I_4|)$, then
        \[
        W_{\D}(I_1,I_3) \asymp \log \frac{\min \{|I_1|, |I_3|\}}{\min\{|I_2|, |I_4|\}} + 1;
        \]
        \item Otherwise,
        \[
        W_{\D}(I_1,I_3) \asymp \left( \log \frac{\min\{|I_2|, |I_4|\}}{\min\{|I_1|, |I_3|\}} + 1 \right)^{-1}.
        \]
    \end{enumerate}
\end{proposition}

In Section \ref{sec:curve-families}, the curve families that appear in our analysis live in a punctured disk.

\begin{lemma}[Canonical rectangles]
\label{lem:canonical-lamination}
    Let $I$, $J_1, \ldots J_k$ be a finite collection of disjoint closed subintervals of $\partial \D$, and let $J = \cup_{i=1}^k J_i$.
    \begin{enumerate}[label=\textnormal{(\arabic*)}]
        \item Let $\mathcal{F}(I,J)$ be the family of all proper curves in the punctured unit disk $\D^*$ that start from $I$ and end at $J$. 
        Then, there exists a finite collection of pairwise disjoint proper rectangles $R_{I,J,1}, R_{I,J,2},\ldots, R_{I,J,l}$ of $\D^*$ such that their vertical leaves are contained in $\mathcal{F}(I,J)$ and
    \[
        \sum_{i=1}^l W(R_{I,J,i}) \geq W(\mathcal{F}(I,J)) - O(1).
    \]
        \item Let $\mathcal{F}(I)$ be the family of all proper curves $\gamma$ in $\D^*$ that have both endpoints in $I$ and are homotopically non-trivial, i.e. for any curve $\gamma'$ in $\mathcal{F}(I)$ that is homotopic to $\gamma$ rel endpoints, $\gamma' \cup I$ separates $0$ from $\infty$. 
        Then, there exists a proper conformal rectangle $R_I$ in $\D^*$ such that its vertical leaves are in $\mathcal{F}(I)$ and
    \[
        W(R_I) \geq W(\mathcal{F}(I)) - O(1).
    \]
    \end{enumerate} 
\end{lemma}

The vertical foliation of the canonical rectangles in this lemma will be referred to as \emph{canonical lamination} associated with $\mathcal{F}(I,J)$ or $\mathcal{F}(I)$.

\begin{proof}
    This follows from lifting to the universal cover and obtain the canonical rectangles out of removing left and right $1$-buffers from to ensure disjointness. 
    Refer to Jeremy Kahn's construction in \cite{K06}.
\end{proof}

\begin{lemma}[Univalent push-forward]
\label{lem:univalent-pushforward}
    Consider a holomorphic self covering map $h :\D^* \to \D^*$ of the punctured unit disk. 
    Let $R$ be a proper rectangle in $\D^*$ and suppose that $h$ is injective on $\partial^{h,0} R$. 
    Then, there exists a subrectangle $R^{\new}$ of $R$ with width $\geq W(R) - 10$ such that $h$ is univalent on the support of $R^{\new}$.
\end{lemma}

\begin{proof}
    The map $h$ is necessarily of the form $z \mapsto e^{2\pi i \tau} z^d$ where $d$ is the degree of $h$ and $\tau$ is some real number.
    Consider the universal covering map $\exp: \UHP \to \D^*$, $\exp(z) = e^{2\pi i z}$.
    The deck group of $\exp$ is generated by unit translation $T_1(z) =z+1$ and the deck group of $h \circ \exp$ is generated by $T_{1/d}(z)= z+1/d$.
    Let $\tilde{R}$ be a lift of $R$ under $\exp$.
    From the hypothesis, 
    \begin{itemize}
        \item $\partial^{h,0} \tilde{R}$ is a real interval of Euclidean length at most $1/d$, and 
        \item the rectangles $\tilde{R}_j := T_1^j(\tilde{R})$, $j \in \Z$ have pairwise disjoint interior.
    \end{itemize}
    We will assume without loss of generality that 
    \[
    \partial^{h,0} \tilde{R}_j < \partial^{h,1} \tilde{R}_j < \partial^{h,0} \tilde{R}_{j+1}.
    \]

    Let us split $\tilde{R}$ into two sub-rectangles $\tilde{S}$ and $\tilde{S}_{\text{far}}$ where 
    \begin{itemize}
        \item vertical leaves of $\tilde{S}$ land at points on the left of the right endpoint of $T_{1/d}(\partial^{h,0} \tilde{R})$;
        \item vertical leaves of $\tilde{S}_{\text{far}}$ land at points on the right of $T_{1/d}(\partial^{h,0} \tilde{R})$.
    \end{itemize}
    The rectangle $\tilde{S}_{\text{far}}$ and its image under $T_{1/d}$ are linked, thus by \cite[Shift Argument]{DL22}, its width is bounded above by $2$. 
    Hence, 
    \[
        W(\tilde{S}) \geq W(\tilde{R}) -  2.
    \]
    Since $\partial^{h,0} \tilde{S}$ is disjoint from $T_{1/d}(\partial^{h,1} \tilde{S})$,
    then removing from $\tilde{S}$ left and right $4$-buffers results in a sub-rectangle $\tilde{R}_{\new}$ that is disjoint from $T_{1/d}(\tilde{R}_{\new})$.
    Therefore, the sub-rectangle $R_{\new} = \exp(\tilde{R}_{\new})$ of $R$ has width at least $W(R) - 10$ and it has the property that $h$ is injective on $R_{\new}$.
\end{proof}

In this paper, we will use the term \emph{lamination} to refer the set of pairwise disjoint arcs that is equal to the vertical leaves of a finite union of rectangles.
We say that a lamination is a \emph{proper} lamination in a surface with boundary if the endpoints of their leaves are on the ideal boundary.
Based on this definition, laminations are subject to the two lemmas above.


\bibliographystyle{alpha}
 
{\small \bibliography{bibliography}}

\end{document}